\documentclass{article}

\usepackage{microtype}
\usepackage{graphicx}
\usepackage{subfigure}
\usepackage{booktabs} %

\usepackage{hyperref}

\usepackage[accepted]{icml2024}
\usepackage{comment}
\usepackage{amsmath}
\usepackage{amssymb}
\usepackage{mathtools}
\usepackage{amsthm}

\usepackage{natbib,bm}

\usepackage[capitalize,noabbrev]{cleveref}

\theoremstyle{plain}
\newtheorem{theorem}{Theorem}[section]

\newtheorem{corollary}[theorem]{Corollary}
\theoremstyle{definition}
\newtheorem{definition}[theorem]{Definition}

\theoremstyle{remark}
\newtheorem{remark}[theorem]{Remark}

\DeclareMathOperator*{\argmin}{argmin}

\usepackage[textsize=tiny]{todonotes}

\icmltitlerunning{Learning to Remove Cuts in Integer Linear Programming}

\begin{document}

\twocolumn[
\icmltitle{Learning to Remove Cuts in Integer Linear Programming}

\icmlsetsymbol{equal}{*}
\begin{icmlauthorlist}
\icmlauthor{Pol Puigdemont}{equal,raclette,paella}
\icmlauthor{Stratis Skoulakis}{equal,raclette}
\icmlauthor{Grigorios G Chrysos}{burger}
\icmlauthor{Volkan Cevher}{raclette}
\end{icmlauthorlist}

\icmlaffiliation{raclette}{LIONS, École Polytechnique Fédérale de Lausanne, Switzerland}
\icmlaffiliation{paella}{Work developed during an exchange coming from Universitat Politècnica de Catalunya (UPC), Spain}
\icmlaffiliation{burger}{Department of Electrical and Computer Engineering, University of Wisconsin-Madison, USA}

\icmlcorrespondingauthor{Pol Puigdemont}{polpuigdemont@gmail.com}

\icmlkeywords{Machine Learning, ICML}

\vskip 0.3in
]

\printAffiliationsAndNotice{\icmlEqualContribution} %

\begin{abstract}
Cutting plane methods are a fundamental approach for solving integer linear programs (ILPs). In each iteration of such methods, additional linear constraints (cuts) are introduced to the constraint set with the aim of excluding the previous fractional optimal solution while not affecting the optimal integer solution. In this work, we explore a novel approach within cutting plane methods: instead of \emph{only} adding new cuts, we also consider the \emph{removal} of previous cuts introduced at any of the preceding iterations of the method under a learnable parametric criteria. We demonstrate that in fundamental combinatorial optimization settings such cut removal policies can lead to significant improvements over both human-based and machine learning-guided cut addition policies even when implemented with simple models.
\end{abstract}

\section{Introduction}
\label{sec: introduction}
Integer linear programming (ILP) has numerous applications in engineering~\cite{MT60}, operational research~\cite{Eiselt}, and finance~\cite{Konno}. In fact, any combinatorial optimization problem can be formulated as an ILP~\citep{conforti2014integer}. ILP is a constrained optimization formulation in which the variables need to take integer values while satisfying some linear constraints. More precisely, an ILP with $n$ variables and $m$ linear constraints admits the form,
\begin{equation}\label{eq:MILP}
z^\star_{\mathrm{int}}:=\min_{x\in \mathbb{R}^n}\{ c^\top x\::\:Ax \leq b,~~x_j\in \mathbb{Z}\},
\end{equation}
where $c\in \mathbb{Z}^n$, $A\in \mathbb{Z}^{m \times n}$ and $b \in \mathbb{Z}^m$.

Despite the fact that ILPs are $\mathcal{NP}$-hard \cite{mip_np_hard}, heuristics, such as the \textit{cutting plane method}~\cite{gomory_cuts} or {branch and bound}~\cite{branch,ZJ0B21} have proven to be extremely valuable.

The \textit{cutting plane method}, proposed by Gomory~\cite{gomory_cuts}, is one of the most fundamental approaches for solving ILPs. The idea is to iteratively solve relaxed versions of the original problem (\ref{eq:MILP}) by dropping the integrality requirement, thus obtaining a linear program (LP) that is computationally tractable to solve. More precisely, by dropping the integrality constraints, we obtain the following LP: 
\begin{equation}\label{eq:lp}
z^\star_{\mathrm{frac}}:=\min_{x\in \mathbb{R}^n}\{c^\top x\::\:Ax \leq b, \:x \in \mathbb{R}^n\}.
\end{equation}
It is clear that $z^\star_{\mathrm{frac}} \leq z^\star_{\mathrm{int}}$. The cornerstone idea of the cutting plane methods is to tighten the bound $z^\star_{\mathrm{frac}}$ by adding cutting planes \cite{gomory_cuts} that increase $z^\star_{\mathrm{frac}}$ but leave $z^\star_{\mathrm{int}}$ unchanged.

Modern solvers \cite{scip8, gurobi} use cutting planes to tighten the bounds on the linear programs (LPs) that are solved iteratively in the cutting plane method \cite{gomory_cuts} and the branch-and-bound method \cite{branch_and_bound_paper}. There are multiple types of cutting planes that can increase the value of the respective LP relaxation. Namely, Gomory showed that for ILPs with $m$ constraints, there are as much as $m$ possible Gomory cuts~\cite{gomory_cuts} (Apendix \ref{subsec:gomory-cuts} contains details on how to obtain Gomory cuts). Although adding all possible cutting planes would yield stronger tightening of the LP relaxation and thus faster convergence, the number of constraints would grow exponentially over the iterations, making the problem infeasible \cite{WesselmannS12}. Consequently, all cutting plane methods select a small number of possible cuts to add at each iteration~\cite{balas,amaldi,andreello,tang2020reinforcement, paulus2022learning}.

In practice, the \textit{cut addition policy}—determining which cut to add at each iteration—plays a crucial role in the convergence properties of the method and over several hand-crafted heuristics have been proposed~\cite{balas,andreello,amaldi,Coniglio}. For instance, the SCIP solver \cite{scip8} employs a weighted sum of features such as integer support and parallelism to rank the cuts. 

A more recent line of research focuses on machine learning-based cut addition policies that, after training, can adjust to the problem distributions of interest (see \cite{summary_LearningCuts} for a survey). These approaches leverage different machine learning techniques, including imitation learning \cite{paulus2022learning}, reinforcement learning \cite{wang2023learning, tang2020reinforcement}, and multiple instance learning \cite{huang2021learning}. It is worth noting that irrespective of the specific training strategy, all previous approaches concentrate on cut addition policies that, at each iteration, include a small set of new cuts in the constraints.

\subsection{Our Approach and Results} 

In this work, we explore cutting plane methods with learning-guided policies that do not only add cuts but also possess the capability to remove them. To the best of knowledge, this is the first work examining such cut removal policies in the context of cutting plane methods.

As already mentioned, adding multiple cutting planes at each iteration of the cut selection method is very beneficial with respect to convergence since larger sizes of the fractional polytope are removed. However, the latter leads to an exponential increase in the number of linear constraints~\cite{WesselmannS12}.

We consider the concept of \textit{cut removals} as a strategy to mitigate the exponential growth in the number of constraints, while still capitalizing on the rapid convergence rates associated with incorporating multiple cuts~\cite{WesselmannS12}. Specifically, in each iteration of the cutting plane method, we opt to include all potential Gomory cuts in the set of linear constraints. As previously noted, this approach provides the advantage of excluding larger regions from the feasibility polytope. Subsequently, we proceed with the cut removal step, where previously introduced cuts are eliminated from the set of linear constraints. This ensures that the overall number of linear constraints increases by just one cut (or a small constant number) from iteration to iteration.

\textbf{Cut Addition vs Cut Removal} As mentioned earlier, our approach differs significantly from previous \textit{cut addition} policies \cite{tang2020reinforcement,paulus2022learning}. Unlike these policies, our method introduces multiple cuts at each iteration but also removes multiple cuts to achieve a balanced increase. In contrast, cut addition policies introduce a small number of new cuts in the constraints, leading to a gradual increment in the number of cuts with each iteration. 

A fundamental distinction between the two approaches lies in the permanent addition of cuts. In cut addition, once a cut is incorporated into the linear constraints, it persists in all subsequent iterations of the method. However, a cut that may have been effective in the early iterations may lose its relevance as more cuts are introduced. Our cutting removal approach addresses this limitation by continually assessing the effectiveness of each cut. This allows us to replace multiple outdated cuts from early iterations with fresh cuts. 

\smallskip
\smallskip
\noindent \textbf{Learning to Remove Cuts} As in the context of cut addition, the strategy for removing cuts plays a crucial role in the convergence properties of our method. An intuitive measure of cut quality is the difference in the LP bound with and without the cut~\cite{Coniglio,paulus2022learning}. Therefore, a straightforward cut removal strategy is to eliminate cuts with the small difference in the LP bound difference. However computing the LP difference for all candidate cuts requires solving numerous LPs which leads to a significant computational overhead.

To address this challenge, we adopt an imitation learning approach, similar to that of \cite{paulus2022learning} in the case of cut addition. Specifically, we train a simple model that uses hand-crafted features of the cuts~\cite{Achterberg2007,WesselmannS12} and an additional MLP layer to predict the LP bound difference for each candidate cut.

Interfacing with an implementation of the Cutting Plane method with Gomory Cuts \cite{gomory_cuts} we experiment on five families of MILPs by training the neural networks to predict the quality of various cuts and remove them acoordingly. Our experimental evaluations indicate that our algorithm outperforms cut addition strategies.

\subsection{Related Work}

There is a recent line of works examining the applications of machine learning techniques in cutting plane methods. \citet{RaduBMT18} were the first to employ machine learning models for predicting the bound improvement of cuts in the context of semi-definite programming. In the realm of Integer Linear Programming, \citet{tang2020reinforcement} utilize reinforcement learning to devise cut addition strategies for Gomory cuts. \citet{huang2021learning} develop alternative cut addition strategies through multiple instance learning. \citet{paulus2022learning} employ imitation learning to create efficient cut selection policies, approximating the computationally expensive \textit{look-ahead policy} that calculates the LP bound improvement for each candidate cut. In a related but slightly orthogonal task, \citet{wang2023learning} leverage the RL framework to determine the order of cuts presented in the LP solver, minimizing the solving time of the LP. \citet{summary_LearningCuts} provides a great survey in recent work for learning to add cuts in the cutting plane method.

At the same time machine learning techniques have been used in the context of other approaches of ILPs such as \textit{Column Generation}~\cite{CKWS22} and \textit{Branch$\&$Bound}~\cite{CKGLP21,GGKM0B20,KBSND16,K16,ZJ0B21}. On the theoretical front \cite{BPSV22,BalcanPSV22,BalcanPSV21} study sample complexity and generalization bounds for cutting plane methods.

\section{Background and Preliminaries}
\label{sec:background}
We denote with $\mathbb{R}$ the set of real numbers and with $\mathbb{Z}$ the set of integer numbers. To simplify notation, we describe the linear constraints ${Ax \leq b}$ with a set of hyperplanes $\mathcal{H}$. More precisely, each hyperplane $(\alpha, \beta) \in \mathcal{H}$ (where $\alpha \in \mathbb{Z}^n$ and $\beta \in \mathbb{Z}$) denotes the hyperplane $\alpha^\top x \leq \beta$.

With a slight abuse of notation, we will write that an $n$-dimensional vector $x\in \mathcal{H}$ if and only if $\alpha^\top x \leq \beta$ for all $(\alpha,\beta)\in \mathcal{H}$. We also denote with $(\mathcal{H},c)$ an instance of Integer Linear Program, $\min_{x\in \mathbb{Z}^n}\{ c^\top x~:~x \in \mathcal{H}\}$. 

We denote with $x^\star_{\mathrm{frac}}:= \mathrm{argmin}_{x \in \mathbb{R}^n}\{c^\top x:x \in \mathcal{H}\}$
the optimal fractional solution of $(\mathcal{H},c)$ and with $x^\star_{\mathrm{int}}:= \mathrm{argmin}_{x \in \mathbb{Z}^n}\{c^\top x:x \in \mathcal{H}\}$ the optimal integral solution. We remark that $x^\star_{\mathrm{frac}}$ can be computed in polynomial time~\cite{linear}. On the contrary computing $x^\star_{\mathrm{int}}$ is an $\mathcal{NP}$-hard problem.

Notice that
$c^\top x^\star_{\mathrm{frac}} \leq c^\top x^\star_{\mathrm{int}}$ and thus solving the fractional problem provides a lower bound on the cost of optimal integral solution. In Definition~\ref{d:cuts} we introduce the notion of cutting plane $(\alpha,\beta)$ that will be crucial throughout the paper.

\begin{definition}\label{d:cuts}
Let the instance $(\mathcal{H},c)$. A hyperplane $(\alpha,\beta)$ is called \textit{cutting plane} if and only if
\[\alpha^\top x_{\mathrm{frac}}^\star > \beta~~~~ \text{ and }~~~~ \alpha^\top x_{\mathrm{int}}^\star \leq  \beta \]
\end{definition}
In case $(\alpha,\beta)$ is a cutting plane, then the new instance $\left(\mathcal{H} \cup (\alpha,\beta),c\right)$ admits a higher optimal fractional value but the same integral optimal value as $(\mathcal{H},c)$. More precisely,
 \[c^\top x_{\mathrm{frac}}^\star \leq c^\top \hat{x}_{\mathrm{frac}}^\star \leq c^\top x_{\mathrm{int}}^\star \]
where $\hat{x}^\star_{\mathrm{frac}}:= \mathrm{argmin}_{x \in \mathbb{R}^n}\{c^\top x:x \in \mathcal{H} \cup (\alpha,\beta)\}$.

\begin{algorithm}[b]
\caption{Cutting plane method}
\begin{algorithmic}[1] 
    \STATE \textbf{Input:} An integer linear program $(\mathcal{H},c)$    
\smallskip
\STATE Set $\mathcal{P}_{1} := \varnothing$
\smallskip
\FOR{ $k=1, \ldots,$}
    \STATE Compute the \textit {fractional solution} of $(\mathcal{H} \cup \mathcal{P}_k,c)$
        \[x_k^\star:= \mathrm{argmin}_{x\in \mathbb{R}^n}\{c^\top x ~:~ x \in \mathcal{H} \cup \mathcal{P}_k\}\]

\STATE \textbf{if} $x_{k}^\star$ is integral ~~\textbf{return} $x_{k}^\star$.

    \smallskip
    \STATE Compute cutpool $\mathcal{C}_k$ and pick $(\alpha_k,\beta_k) \in \mathcal{C}_k$. 
\smallskip
    \STATE Set $\mathcal{P}_{k+1} := \mathcal{P}_{k} \cup (\alpha_k,\beta_k)$ \textcolor{blue}{\textit{$\#$ insert new constraint}}
\ENDFOR
\label{alg:CP}
\end{algorithmic}
\end{algorithm}

In his seminal work, \citet{gomory_cuts} showed that if $(\mathcal{H}, c)$ admits a strictly fractional solution $x_{\mathrm{frac}}^\star$ then there exists a set of separating hyperplanes $\mathcal{C}$, i.e. each hyperplane $(\alpha,\beta)\in \mathcal{C}$ is a separating hyperplane. 

\begin{theorem}\label{t:cuts}[\citet{gomory_cuts}] Let an instance $(\mathcal{H},c)$ with a strictly fractional solution $x_{\mathrm{frac}}^\star$. Then there exists a set of hyperplanes $\mathcal{C}$ with $|\mathcal{C}| = |\mathcal{H}|$ such that any
$(\alpha, \beta) \in \mathcal{C}$ is a cutting plane. $\mathcal{C}$ is also referred as cutpool.

\end{theorem}
\citet{gomory_cuts} proposed the seminal \textit{cutting plane method} (Algorithm~\ref{alg:CP}) that solves ILPs by iteratively adding new cutting planes to the constraints.

The cornerstone idea of the cutting plane method is that in case the solution $x_k^\star$ at iteration $k$ is not integral then the cutting plane $(\alpha_k,\beta_k)$ added at Step~$6$ will render $x_k^\star$ infeasible while keeping the optimal integer solution $x^\star_{\mathrm{int}}$ of $(\mathcal{H},c)$ intact. As a result, Algorithm~\ref{alg:CP} guarantees that
$c^\top x_{k+1}^\star \geq c^\top x_{k}^\star$ until a final integral solution is reached. 
\smallskip

\textbf{Cut Selection} We remark that regardless the way $(\alpha_k,\beta_k)$ is selected at Step~$6$, Algorithm~\ref{alg:CP} is always guaranteed to converge to the optimal integral solution. However in practice the cut selection policy plays a major role on the convergence properties of the cutting plane method. As a result, over the years various handcrafted heuristics have been proposed~\cite{balas,andreello,amaldi,Coniglio}.

One of the most natural heuristics to select cutting planes is the one that directly \textit{looks ahead} on the improvement of the fractional LP bound~\cite{amaldi,Coniglio,paulus2022learning}. As in \citet{paulus2022learning} we will refer to this as the \textit{look-ahead policy}.  More precisely,
\noindent
\small
\[(\alpha_k,\beta_k):= \mathrm{max}_{\textcolor{blue}{(\alpha,\beta)}\in \mathcal{C}_k}\left[ \min_{x} \{c^\top x: x \in \mathcal{H}_k \cup \mathcal{P}_k\cup\textcolor{blue}{(\alpha,\beta)}\right]. \]
\normalsize

On the positive side, \textit{look-ahead policy} admits very favorable convergence properties to the optimal integral solution~\cite{paulus2022learning}. On the negative side, implementing the above cut selection policy is very time-consuming since it requires the solution of $|C_k|$ linear programs at each iteration $k$ which creates a significant computational overhead. 
\paragraph{Evaluating Cut Selection Policies}
We can utilize the number of iterations until convergence to the optimal integer solution as a measure of performance of a cutting plane method. However cutting plane methods can take a lot of iterations before converging to the the optimal integral solution, rendering the latter metric not a very practical~\cite{tang2020reinforcement,paulus2022learning}. 

An alternative performance metric is to measure how far is the solution of method at iteration $k$ is from the optimal integer solution. The latter is referred as \textit{integrality gap}~\cite{tang2020reinforcement,paulus2022learning} which at iteration $k$ is defined as
$g_k:= c^\top x_{\mathrm{int}}^\star - c^\top x_{k}^\star \geq 0$. 

Different MILP problems, even from the same distribution, will have different magnitudes for $g_k$. To compare different instances we can measure the factor by which the integrality gap is closed between the first relaxation and the current round. At iteration $k$ of the CP method, the \emph{integrality gap closure} (IGC) \cite{tang2020reinforcement,paulus2022learning} is defined as
\begin{equation} \label{eq:igc}
    IGC_k := \frac{g_1 - g_k}{g_1} = \frac{c^\top x_k^\star - c^\top x_1^\star}{c^\top x_{\mathrm{int}}^\star - c^\top x_1^\star} \in [0,1].
\end{equation}

\section{Cut Removal Policies}
\label{sec:methodology}
The starting point of this work concerns the design of novel cutting plane methods that, at each iteration, not only add cuts but are also capable of removing cuts. As we shall see shortly, the latter offers significant advantages compared to cutting plane methods that only add cuts at each iteration.

\textbf{Adding Multiple Cuts} Notice that at each iteration of Algorithm~\ref{alg:CP}, only one new cutting plane $(\alpha_k, \beta_k)$ is added. However, it would make perfect sense at iteration $k$ to include the entire cutpool $\mathcal{C}_k$ of Theorem~\ref{t:cuts}. The latter would result in an even greater increase in the objective function and, consequently, better convergence properties. However since $|\mathcal{C}_k|$ can have size up to $|\mathcal{H}|+|\mathcal{P}_k|$, the latter would lead to an exponential increase in the number of constraints, rendering the solution of the linear program at Step~$4$ impossible. This the reason that previous cutting plane methods have focused on adding just one or small constant number of cutting planes at each iteration~\cite{tang2020reinforcement,huang2021learning,paulus2022learning}.

\textbf{Cut Removal} Our approach revolves around addressing the exponential surge resulting from the inclusion of multiple cuts in each iteration. We incorporate an additional cut-removal procedure to prevent the number of constraints from escalating exponentially. In Algorithm~\ref{alg:CP2}, we outline the fundamental pipeline of our approach.
\begin{algorithm}
\caption{Cutting Plane method with Cut Removal}
\begin{algorithmic}[1] 
    \STATE \textbf{Input:} An integer linear program $(\mathcal{H},c)$    
\smallskip
\STATE Set $\mathcal{P}_{1} := \varnothing$
\smallskip

\FOR{ $k=1, \ldots,$}
    \STATE Compute the cutpool $\mathcal{C}_k$ for $(\mathcal{H}\cup \mathcal{P}_k,c)$.

\smallskip
        
    \STATE Compute the \textit {fractional solution} of $(\mathcal{H}\cup \mathcal{P}_{k} \cup \mathcal{C}_k,c)$
        \[x_k^\star:= \mathrm{argmin}_{x\in \mathbb{R}^n}\{c^\top x : x \in \mathcal{H}\cup \mathcal{P}_{k} \cup \mathcal{C}_k\}\]
    \textit{\textcolor{blue}{$\#$ include all possible cuts}}
    \smallskip
    \smallskip
   \STATE \textbf{if} $x_{k}^\star$ is integral ~~\textbf{return} $x_{k}^\star$.
    
    \smallskip
        \smallskip

    \STATE Select a subset $\hat{P}_{k+1} \subseteq \mathcal{P}_k \cup \mathcal{C}_k$ with $|\hat{P}_{k+1}| = k+1$.\\
    \textit{\textcolor{blue}{$\#$ removal of cuts}}
    \smallskip
    \smallskip
   \STATE $\mathcal{P}_{k+1} := \hat{P}_{k+1} \cup \{ c^\top x \geq \lceil c^\top x_k^\star \rceil\}$
   \\
    \textit{\textcolor{blue}{$\#$ new constraints for the next iteration}} 
 \smallskip    
\ENDFOR
\label{alg:CP2}
\end{algorithmic}
\end{algorithm}

Let us explain how the Algorithm~\ref{alg:CP2} is able to combine both the advantages of selecting multiple cuts at each iteration while at the same time avoiding the exponential increase in the number of linear constraints.

In Step~$4$, Algorithm~\ref{alg:CP2} computes the cutpool $\mathcal{C}_k$ (see Definition~\ref{d:cuts}) for the problem $(H \cup \mathcal{P}_k,c)$ that includes the original constraints $\mathcal{H}$ plus the additional cuts $\mathcal{P}_k$ that have been added so far.

In Step 5, Algorithm~\ref{alg:CP2} calculates the optimal solution $\hat{x}_k$ for the instance $(H \cup \mathcal{P}_k \cup \mathcal{C}_k,c)$
meaning that the algorithm completely incorporates all the cutpool $\mathcal{C}_k$. To this end we remark that due to Definition~\ref{d:cuts}, the optimal integral solution of $(\mathcal{H},c)$ is the same with the optimal integral solution of $(\mathcal{H}\cup \mathcal{P}_k \cup \mathcal{C}_k ,c)$. Namely,
\[\argmin_{x\in \mathbb{Z}^n}\{c^\top x:x \in \mathcal{H}\} = \argmin_{x\in \mathbb{Z}^n}\{c^\top x:x \in \mathcal{H}\cup \mathcal{P}_k \cup \mathcal{C}_k\}.\]
The idea behind this step is to achieve a substantial improvement in the objective function compared to the strategy of introducing just one cut $(\alpha,\beta) \in \mathcal{C}_k$. In other words,
\noindent
\small
\[\min_{x\in \mathbb{R}^n}\{c^\top x:x \in \mathcal{H}\cup \mathcal{P}_k \cup \mathcal{C}_k\} \geq \min_{x\in \mathbb{R}^n}\{c^\top x:x \in \mathcal{H}\cup \mathcal{P}_k \cup (\alpha,\beta)\}\]
\normalsize
where the last inequality comes from the fact that the left problems admit a superset of constraints with respect to the right one.

In Step 6 of Algorithm~\ref{alg:CP2}, the cut removal process is executed to prevent an exponential rise in the number of constraints. To be more specific, only $k+1$ cuts from the set of previous cuts, $\mathcal{P}_k \cup \mathcal{C}_k$, are kept. This ensures that only one additional cut is introduced in each iteration, preserving the tractability of the solution obtained in Step 4 of the linear programming formulation.

Step 7 plays a crucial role in ensuring the algorithm's monotonic improvement. Note that $\mathcal{P}_k$ may not necessarily be a subset of $\hat{\mathcal{P}}_k$, implying that a cut introduced in a previous iteration can be removed in a subsequent iteration if it becomes uninformative. The positive aspect of this approach lies in the ability to replace inactive cuts from earlier iterations with more valuable ones. On the downside, as $\mathcal{P}_k$ is not guaranteed to be a subset of $\hat{\mathcal{P}}_{k+1}$, there is a potential for the objective function to decrease. To mitigate this, a special constraint $c^\top x \geq \lceil c^\top x_k^\star \rceil$ is included. This constraint ensures that the fractional value of the problem $(\mathcal{H}\cup \mathcal{P}_{k+1}, c)$ can only increase with respect to the fractional value of the problem $(\mathcal{H}\cup \mathcal{P}_{k}, c)$.  
\begin{remark} The additional special constraint $c^\top x \geq \lceil c^\top x_k^\star \rceil$ does not affect the optimal integral solution. The latter is formally stated and proven in Corollary~\ref{c:1}.
\end{remark}

\begin{corollary}\label{c:1}
It holds that
$\argmin_{x\in \mathbb{Z}^n}\{c^\top x:x \in \mathcal{H}\cup \mathcal{P}_{k+1}\} = \argmin_{x\in \mathbb{Z}^n}\{c^\top x:x \in \mathcal{H}\}$
\end{corollary}
\begin{proof}
The reason is that the following holds
\begin{eqnarray*}
\min_{x\in \mathbb{Z}^n}\{c^\top x:x \in \mathcal{H}\} &=& \min_{x\in \mathbb{Z}^n}\{c^\top x:x \in \mathcal{H}\cup \mathcal{P}_k\}\\
&=&\min_{x\in \mathbb{Z}^n}\{c^\top x:x \in \mathcal{H}\cup \mathcal{\hat{P}}_{k+1}\}\\
&\geq& \underbrace{\min_{x\in \mathbb{R}^n}\{c^\top x:x \in \mathcal{H}\cup \mathcal{\hat{P}}_{k+1}}_{c^\top x_k^\star}\}
\end{eqnarray*}
The first equality holds inductively and the second comes from the fact that $\hat{\mathcal{P}}_{k+1} \subseteq \mathcal{H} \cup \mathcal{P}_k$. Notice that since $c \in \mathbb{Z}^n$, $\min_{x\in \mathbb{Z}^n}\{c^\top x:x \in \mathcal{H}\}\geq \lceil c^\top x_{k}^\star \rceil$. Thus
$\min_{x\in \mathbb{Z}^n}\{c^\top x:x \in \mathcal{H} \cup \mathcal{\hat{P}}_{k+1}\}\geq \lceil c^\top x_k^\star \rceil$. As a result, adding the additional constraint $c^\top x \geq  \lceil c^\top x_k^\star \rceil$ does not affect the optimal integral solution.
\end{proof}

\subsection{Learning to Remove Cuts} In this section, we will examine what constitutes a reasonable cut removal criterion for Step 6 of Algorithm~\ref{alg:CP2} that can contribute to fast convergence properties. Motivated by the \textit{look-ahead} policy presented in Section~\ref{sec:background}, the most intuitive cut-removal criterion is the one that maximizes the objective function. Namely:
\[\hat{P}_{k+1}:=\mathrm{argmax}_{\textcolor{blue}{\mathcal{P} \subseteq \mathcal{P}_k \cup \mathcal{C}_k}} \left[c^\top x:x \in \mathcal{H}\cup 
 \textcolor{blue}{\mathcal{P}}\right].\]
Unfortunately, there is no efficient algorithm for solving the problem above. Consequently, the most natural approach is to consider the \textit{greedy look-ahead criterion} where we keep the $k+1$ cuts of $\mathcal{P}_k \cup \mathcal{C}_k$ that lead to highest bounds in the objective function and remove the rest. More precisely, we associate each cutting plane $(\alpha, \beta) \in \mathcal{P}_k \cup \mathcal{C}_k$ with the following score,
\begin{eqnarray}\label{eq:1}
\mathrm{SC}(\alpha,\beta)&=& \min_{x\in \mathbb{R}^n}\{c^\top x~:~x\in \mathcal{H}\cup \mathcal{P}_k \cup \mathcal{C}_k\}\\
&-& \min_{x\in \mathbb{R}^n}\{c^\top x~:~x\in \mathcal{H}\cup \mathcal{P}_k \cup \mathcal{C}_k/ (\alpha,\beta)\},\nonumber
\end{eqnarray}
which measures the difference of the fractional LP value once $(\alpha,\beta) \in \mathcal{P}_k \cup \mathcal{C}_k$ is removed. Then $\hat{\mathcal{P}}_{k+1}$ is defined as the $k+1$ cuts of $\mathcal{P}_k \cup \mathcal{C}_k$ with the highest score.

\textbf{Parametrizing Cut Removal}
Unfortunately, similar to the case of single-cut addition, the cut removal policy of Eq.~(\ref{eq:1}) is highly computationally demanding to implement since it requires solving numerous linear programs. To circumvent this additional computational overhead, we will approximate the scores of Eq.~(\ref{eq:1}) through an adequately parametrized model $\pi_\theta(\cdot)$. More precisely, for each cut $(\alpha, \beta)\in \mathcal{P}_k \cup \mathcal{C}_k$ we compute the model $\pi_\theta(\cdot)$ to compute the scores
\[
\mathrm{SC}(\alpha,\beta)= \pi_{\theta}(\mathcal{H},\mathcal{P}_k,\mathcal{C}_k, (\alpha,\beta)).\]
More precisely, we first convert each cut $(\alpha,\beta)$ into a $14$-dimensional feature vector. Due to the fact that the coordinate of the feature vectors involve also the optimal fractional solution $x_k^\star := \min \{c^\top x: x\in H \cup \mathcal{P}_k \cup \mathcal{C}_k \}$ these feature embeddings naturally encode the overall structure of the linear constraints. The details of feature vectors can be found in Appendix~\ref{subsec:feature-encoding}.  The encoded state is then fed to a Multi Layer Perceptron (MLP) with a sigmoid ($\sigma$) activation.

A summary of our approach is illustrated in Figure~\ref{fig:cutremoval}.

\begin{figure}[t]
  \vskip 0.2in
  \begin{center}
    \includegraphics[width=0.91\columnwidth]{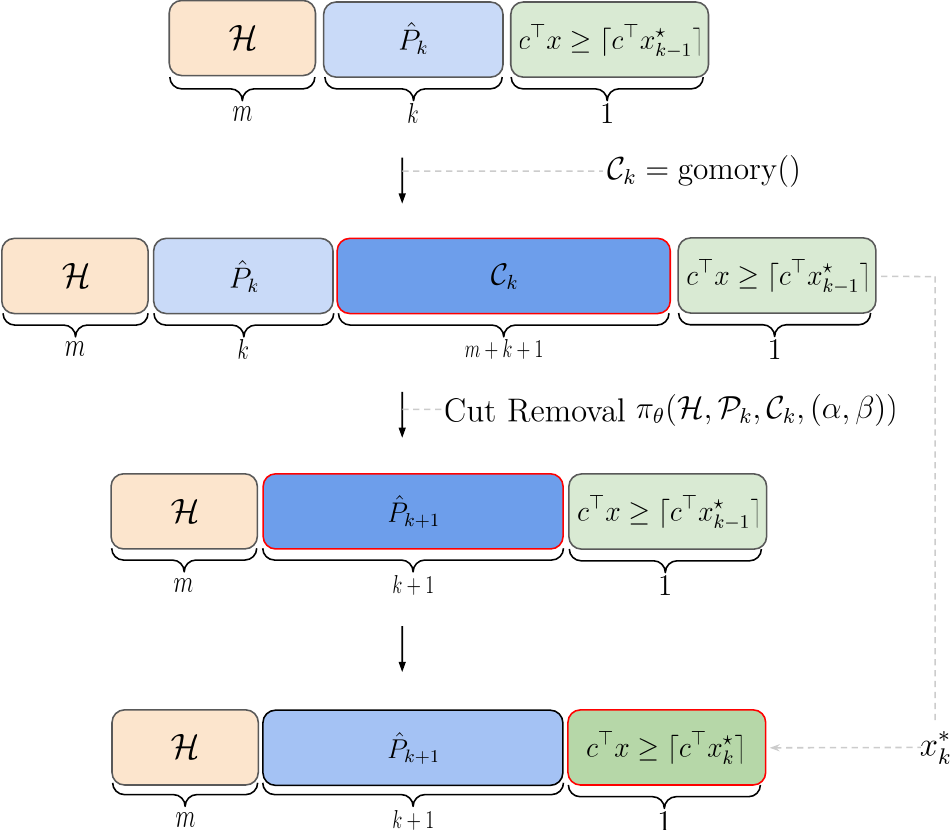}  
    \caption{Visual representation of Algorithm \ref{alg:CP2}.}
    \label{fig:cutremoval}
  \end{center}
  \vskip -0.2in
\end{figure}

\begin{remark}
More complicated models such as graph neural networks over a graph encoding of the state as in \citet{paulus2022learning, gasse} could also be used. Although graph architectures draw an interesting line of improvements they are outside the scope of this study.
\end{remark}

\subsection{Training the Model}
The goal of training process is to approximately select the parameters $\theta$ so as to predict the scores of Eq.~(\ref{eq:1}). In other words select $\theta$ such that
\begin{eqnarray*}\label{eq:1-approximation}
&&\pi_{\theta}(\mathcal{H},\mathcal{P}_k,\mathcal{C}_k, (\alpha,\beta))\simeq \min_{x}\{c^\top x~:~x\in \mathcal{H}\cup \mathcal{P}_k \cup \mathcal{C}_k\}\\
&&- \min_{x}\{c^\top x~:~x\in \mathcal{H}\cup \mathcal{P}_k \cup \mathcal{C}_k/ \{\alpha^\top x \leq \beta\}\}.\nonumber
\end{eqnarray*}   
We note that \citet{paulus2022learning} adopt a similar imitation learning approach to approximate the scores of the \textit{look-ahead policy} presented in Section~\ref{sec:background}.

\paragraph{Dataset Generation}
In order to generate our training data we use trajectories produced by the cutting plane method with the \textit{look-ahead policy} for various classes of ILPs of such as Packing, Bin Packing, Set Cover, Max Cut and Production Planning.

For each intermediate ILP produced at iteration $k$, denoted as $(\mathcal{H} \cup \mathcal{P}_k, c)$, we utilize the corresponding cut pool $\mathcal{C}_k$ and the previously added cuts $\mathcal{P}_k$ to generate the following data set. For every cutting plane $(\alpha, \beta) \in \mathcal{C}_k \cup \mathcal{P}_k$, a new data point is created with a value equal to the \textit{normalized LP improvement},
\[\frac{c^\top x^\star_{k\cup (\alpha,\beta)} - c^\top x^\star_{k}}{c^\top x^\star_{k}} > 0\]
where $x_k^\star:= \argmin_{x}\{c^\top x:x\in \mathcal{H}\cup \mathcal{P}_k\}$ and $x_{k\cup(\alpha,\beta)}^\star:= \argmin_{x}\{c^\top x:x\in \mathcal{P}\cup \mathcal{C}_k\cup (\alpha,\beta)\}$.
The reason that we use the normalized LP improvement over the actual LP improvement $c^\top x^\star_{k\cup (\alpha,\beta)} - c^\top x^\star_{k}$ is to get target values that are not as dependent on the specific instance of the problem family.

Once all training data are generated, we solve a regression problem with quadratic loss to choose the parameters $\theta$.

\section{Experiments} \label{sec:experiments}
We divide our experiments in two main parts. The first one focuses on evaluating the performance of cut removal acting against multiple benchmark policies by rolling them out on synthetic test MILP instances for each of the problem families in a controlled environment. Next, we investigate how well do our trained models generalize to larger instances.

\begin{figure*}[t]
\centering
\includegraphics[width=0.8\textwidth]{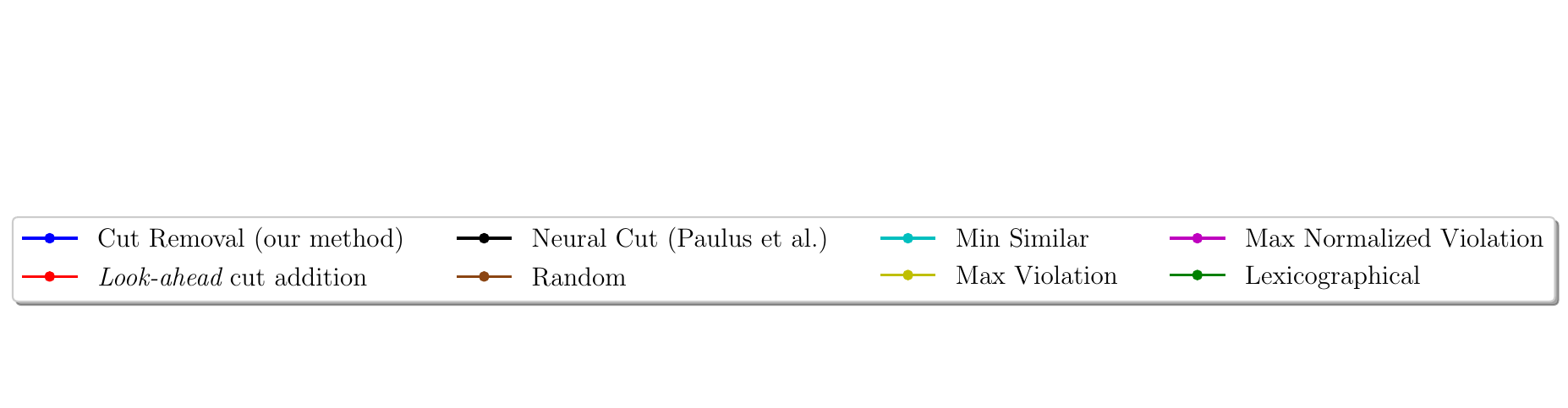}\\
\vspace{-17pt}
\includegraphics[width=0.325\textwidth]{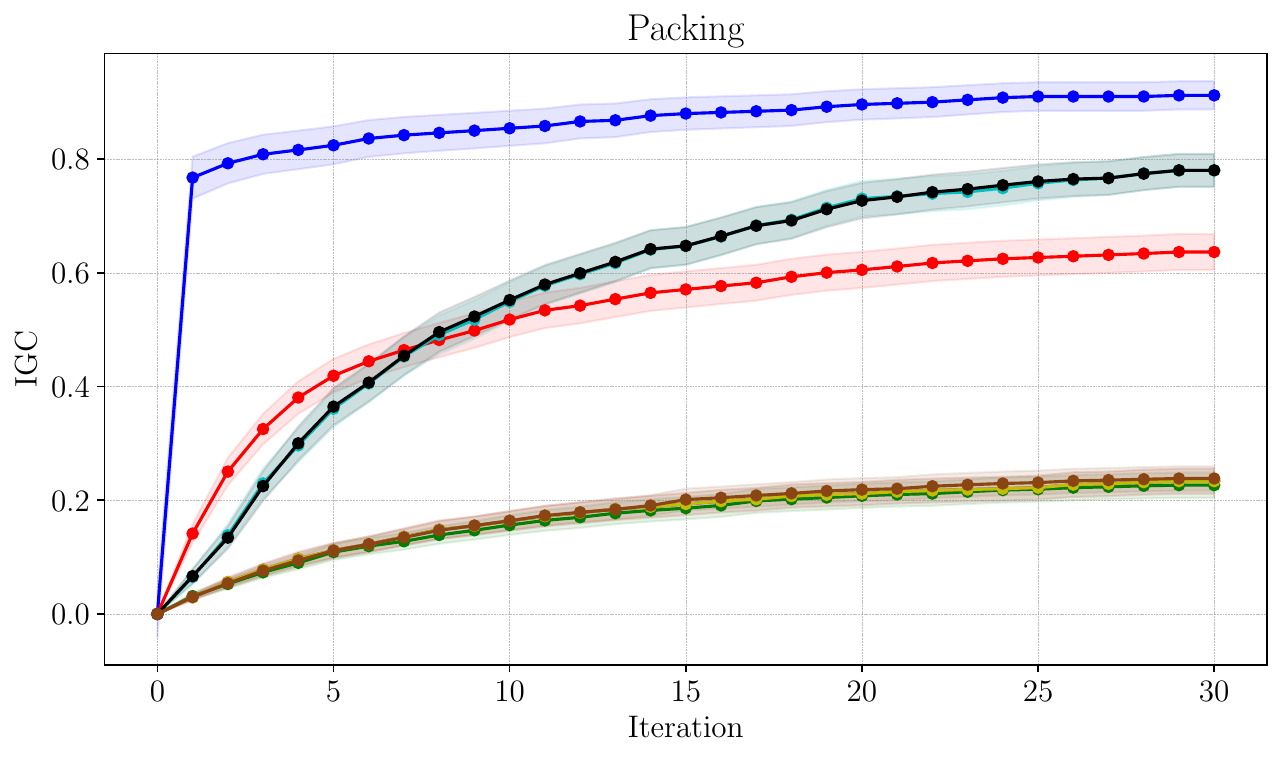}
\includegraphics[width=0.325\textwidth]{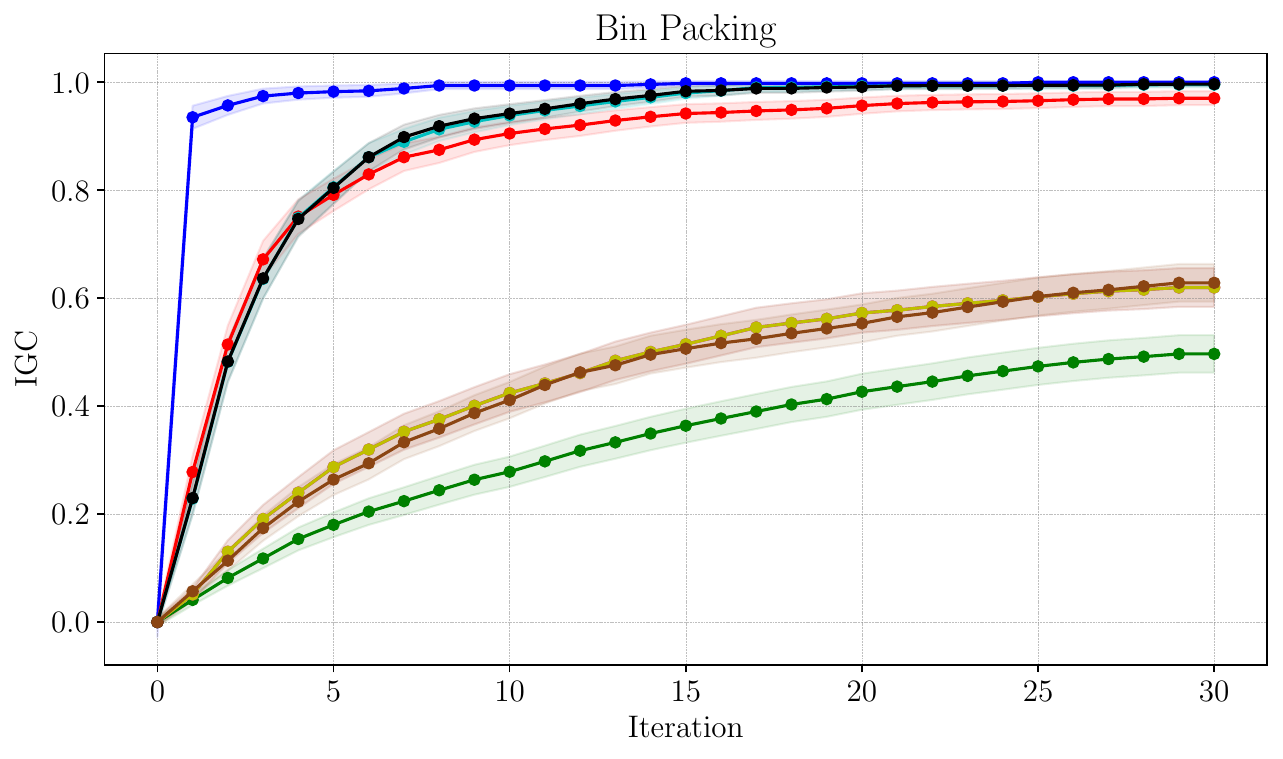}
\includegraphics[width=0.325\textwidth]{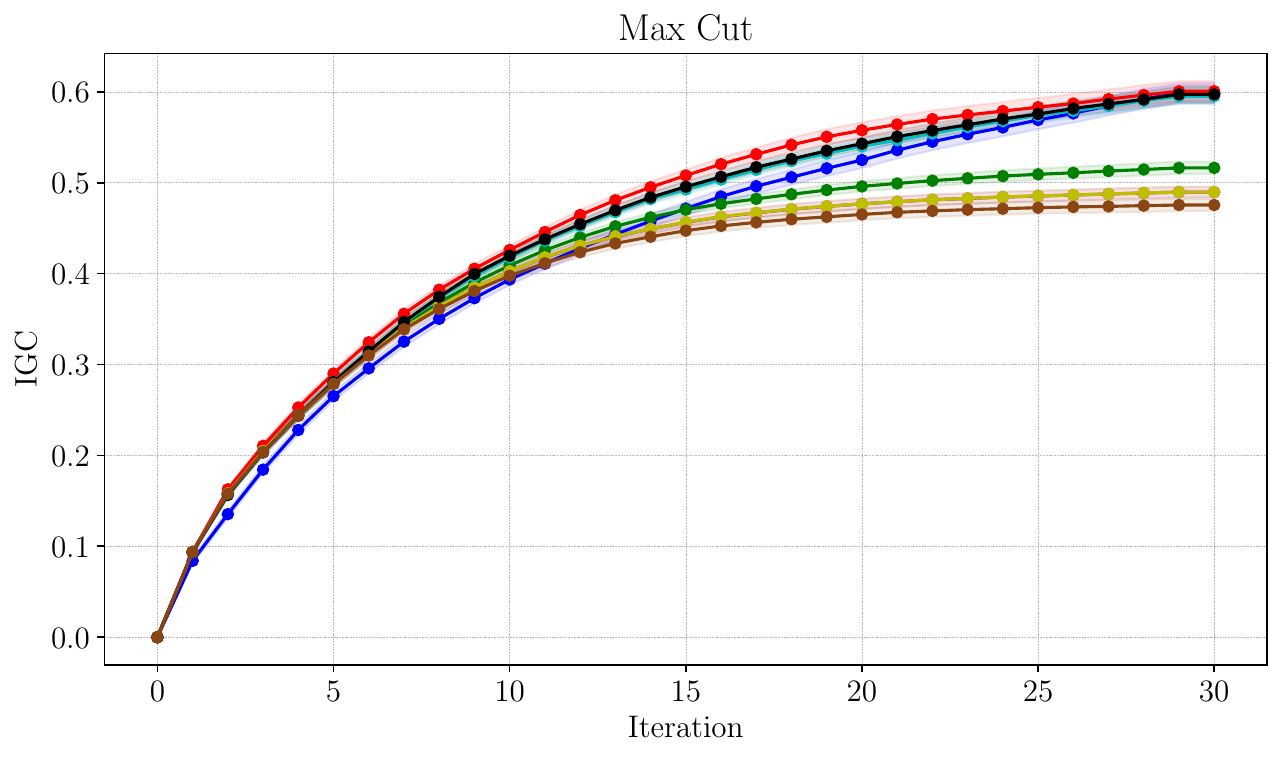}\\
\includegraphics[width=0.325\textwidth]{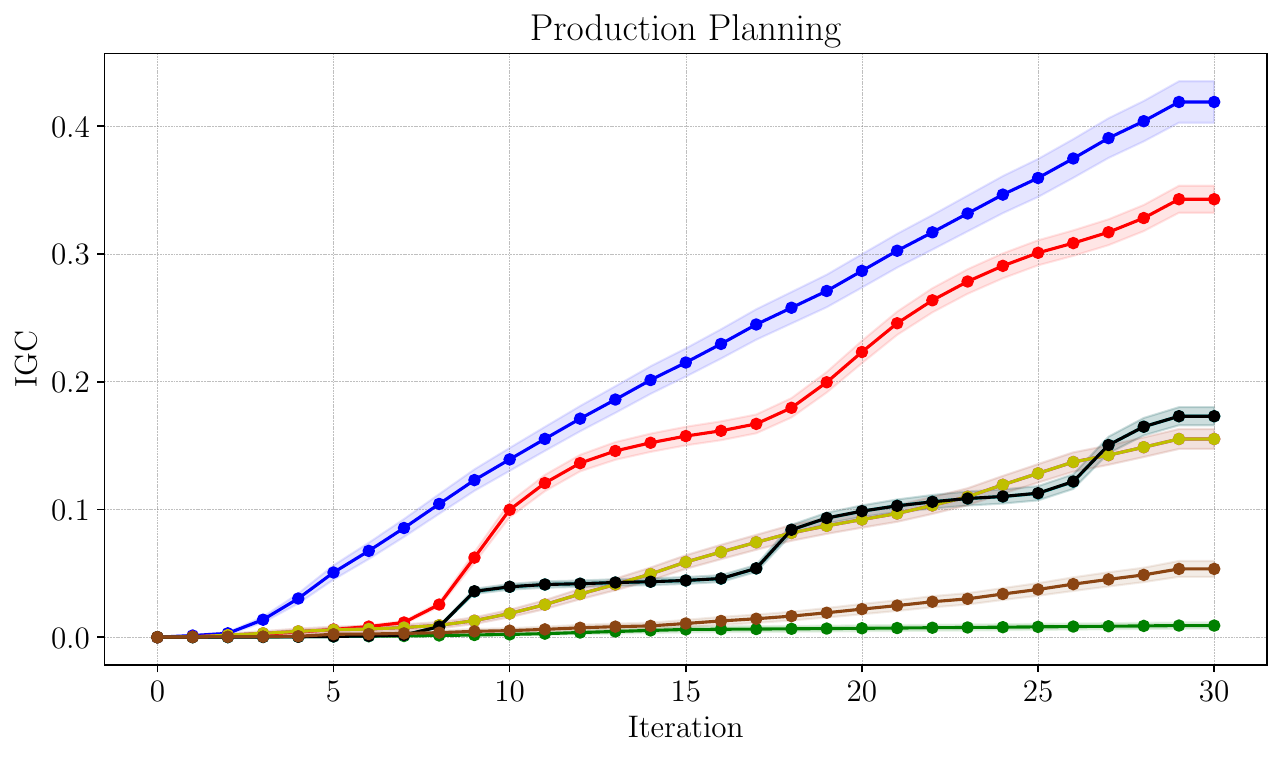}
\includegraphics[width=0.325\textwidth]{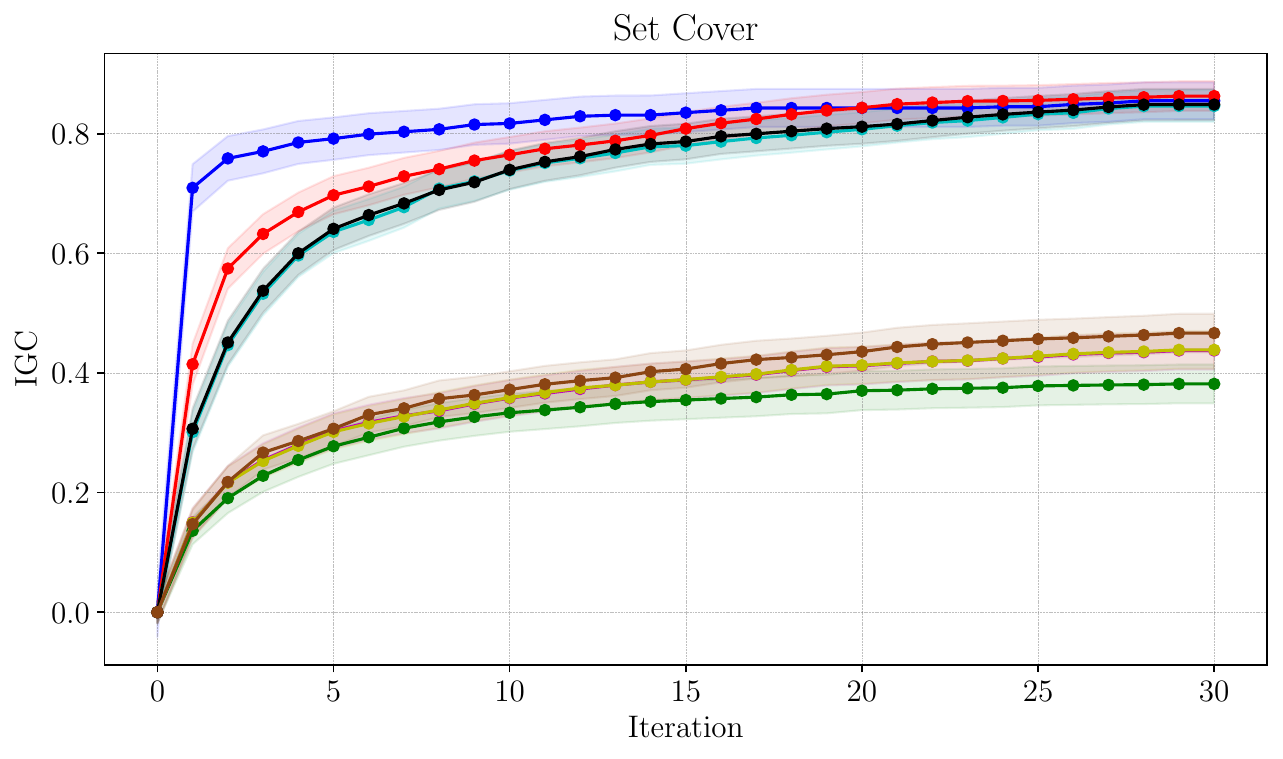}\\
\vspace{-10pt}
\caption{\textbf{Mean IGC for the benchmark instances:} We report the mean Integral Gap Closed (IGC) that measures quality of the solution with respect to the optimal integer solution (see Section~\ref{sec:background} for details), a policy achieving larger IGC values with fewer iterations is better. The highlighted region represents the variance. Our cut removal algorithm outperforms or matches all cut addition methods in all benchmarks. For cut addition policies the $\textit{look-ahead}$ outperforms all of the others except in Packing where Neural Cut, this behavior matches the results showcased in the equivalent benchmark of \citet{paulus2022learning}.}
\label{fig:Experiment_1}
\end{figure*}

\subsection{Evaluation Setup}
As in~\citet{tang2020reinforcement,paulus2022learning}, we assess the performance of the various cutting plane methods by considering the improvement over the Integral Gap Closure value (IGC) (see Section~\ref{sec:background}) with respect to the number of iterations. This evaluation metric offers a robust estimate of the actual running time of the cutting plane methods, as the primary bottleneck in running time arises from solving linear programs (LPs). Simultaneously, this metric provides a cleaner benchmark since running time significantly depends on factors such as implementation, the choice of backbone solver, and available compute resources.

\citet{paulus2022learning} employ imitation learning to train a cut addition policy, mimicking the behavior of the \textit{look-ahead policy} discussed in Section~\ref{sec:background}. Their study demonstrates that the imitation learning approach outperforms the reinforcement learning method proposed by \citet{tang2020reinforcement}. However, both approaches are surpassed by the \textit{look-ahead policy}, which involves solving numerous LPs at each iteration. Unfortunately, neither \citet{paulus2022learning} nor \citet{tang2020reinforcement} provide a public implementation of their code. To ensure a fair comparison, we consider the \textit{look-ahead policy}, which outperforms both previous approaches, and an in-house implementation of the \textit{Neural Cut} method proposed by \citet{paulus2022learning} using the same feature encoding as in our model (see Appendix~\ref{subsec:feature-encoding}). We also include several human-crafted heuristics such as the min similar, max normalized violation, max violation, lexicographical and random (see Appendix~\ref{subsec:baseline-heuristics} for the definitions).

\smallskip

\begin{remark}\label{rem:environment}
\citet{paulus2022learning} utilize the SCIP solver, which incorporates instance pre-solving and various types of cuts~\cite{Achterberg2007}. To ensure a fair benchmarking environment focused solely on cut decisions, we benchmark on an implementation of the Cutting Plane method from scratch. The implementation considers only Gomory Cuts without any pre-solving modes.
\end{remark}

In order to stress-test our environment and compute the optimal solution required to obtain the IGC metrics we use the SCIP solver \cite{scip8}.

Our models are trained with \cite{sgd} using Pytorch \cite{pytorch} for the implementation.

\subsection{Datasets} \label{subsec:dataset}
We experiment with five families of MILPs: packing, bin packing, max cut, production planning (as in \citet{paulus2022learning, tang2020reinforcement}) and set cover. For each family, instances are generated randomly. Packing, bin packing, max cut and production planning are generated under the random formulations used in \citet{tang2020reinforcement, paulus2022learning}. For set cover we suggest our own probabilistic formulation. Details on the generation of the instances can be found in the Appendix \ref{subsec:dataset-dimensions}.

\subsection{In Distribution Evaluation} \label{subsec:core-experiment}
The core experiments in this work aim to respond the following question: Can our cut removal acting algorithm (after training) surpass the $\textit{look-ahead}$, \citet{paulus2022learning} and other cut addition baselines on unseen instances? We answer the previous question positively on the five different ILP benchmarks.

\paragraph{ILP Benchmarks}
We generate a total of $3000$ instances for each of the five problem families. Regarding the dimensions, the small and medium instances suggested in \citet{tang2020reinforcement} were too small for our comparisons and they were being solved after too few iterations to extract differences. For this reason we only consider large instances ($\sim$100 variables, 100 constraints) as done in \citet{paulus2022learning}.

For each problem family, we use $2000$ instances for training, $500$ instances for validation and $500$ instances for testing as done in \citet{paulus2022learning}. We collect trajectories of the $\textit{look-ahead}$ policy for the train and validation instances. We train these parametric models $\pi_\theta$ with these trajectories to predict the bound improvement as previously specified in Section \ref{sec:methodology}. Finally, the various cutting plane methods are compared in terms of their performance on the test instances.

The dimensions (variables, constraints) for training, validation and testing are the same for each problem. The specifications for each problem can be found in Appendix \ref{subsec:dataset-dimensions}.

\paragraph{Test Evaluation Metrics}
In order to assess which method is best we compare the mean IGC after each iteration across all instances as done in \citet{paulus2022learning, tang2020reinforcement}. Larger IGC with fewer iterations denotes better performance. The IGC is $0$ at the start and $1$ in case of convergence to the optimal intergral solution.

We don not compare execution times because the focus of our experiments relies in evaluating how large is the improvement at each round as in \citet{paulus2022learning, tang2020reinforcement}. A time detailed evaluation would depend on the LP solver, pre-solving of the instances, ordering of the constraints \cite{wang2023learning} and other factors which are outside the scope of this evaluation. Our experimental findings are presented in Figure~\ref{fig:Experiment_1}.

\paragraph{Our Experimental Results} 
After training on the test instances, we select the best $\pi_\theta$ according to the validation loss.
Our cut removal algorithm outperforms all cut addition policies in packing, bin-packing, production planning and set cover families. For packing and production planning the mean IGC of our method outpeforms significantly the comparing methods. For bin-packing and set cover the mean IGC is significantly larger in the first half of the iterations. We observe that this match in performance occurs in IGC values close to optimal pointing that in mean the algorithms are able to reach the optimal integral solution before iteration 30.

For max cut cut removal acting does not outperform the expert and neural cut addition methods. Nevertheless, the differences in performance are very small, for instance, a random policy performs as well as the best policies in the first third of the iterations. This behaviour is consistent with the findings in \citet{paulus2022learning}.

\begin{figure*}[t]
\centering
\includegraphics[width=0.8\textwidth]{plots/legend.pdf}\\
\vspace{-17pt}
\includegraphics[width=0.33\textwidth]{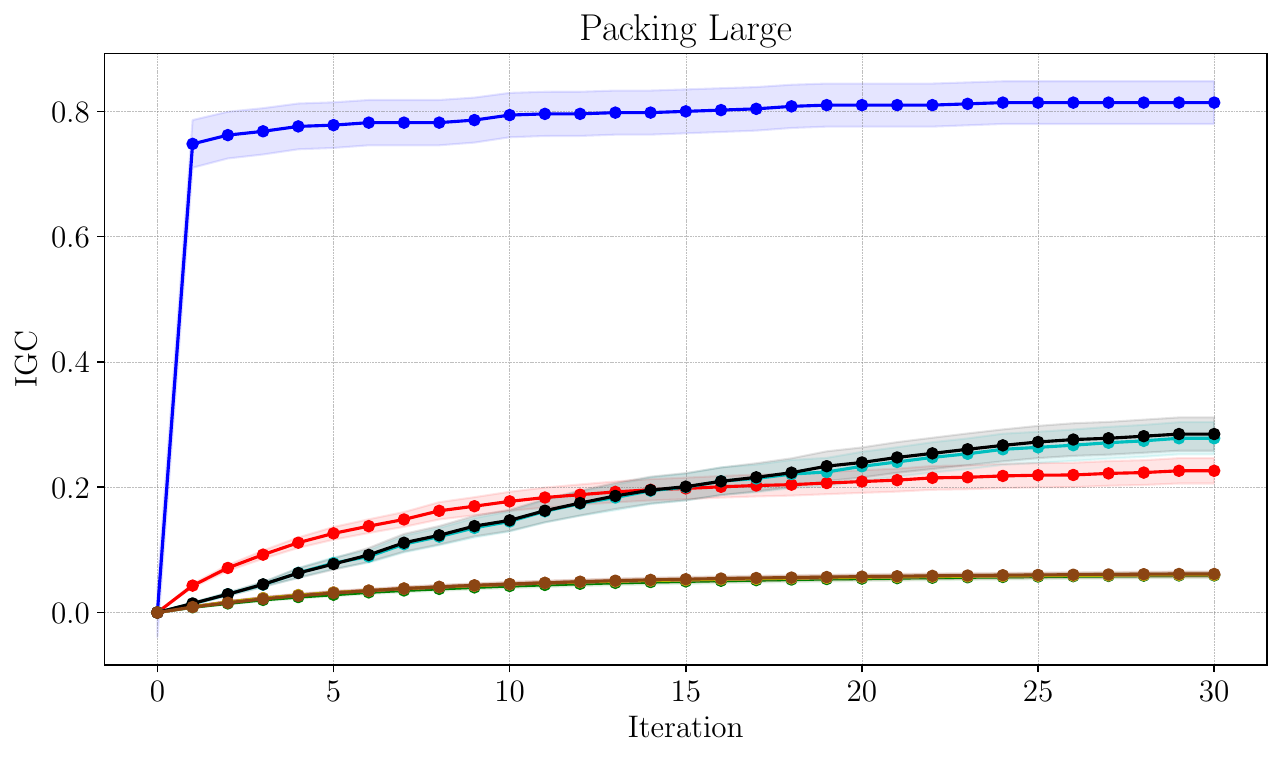}
\includegraphics[width=0.33\textwidth]{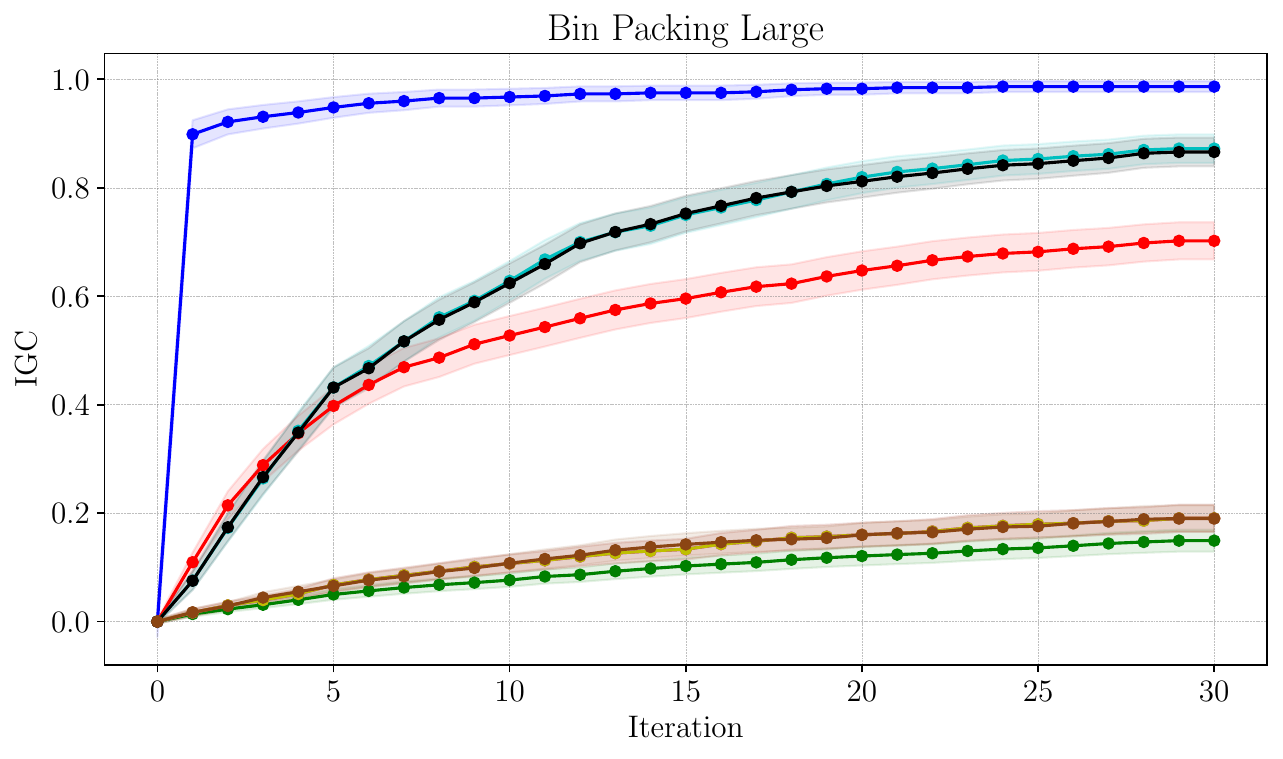}
\includegraphics[width=0.33\textwidth]{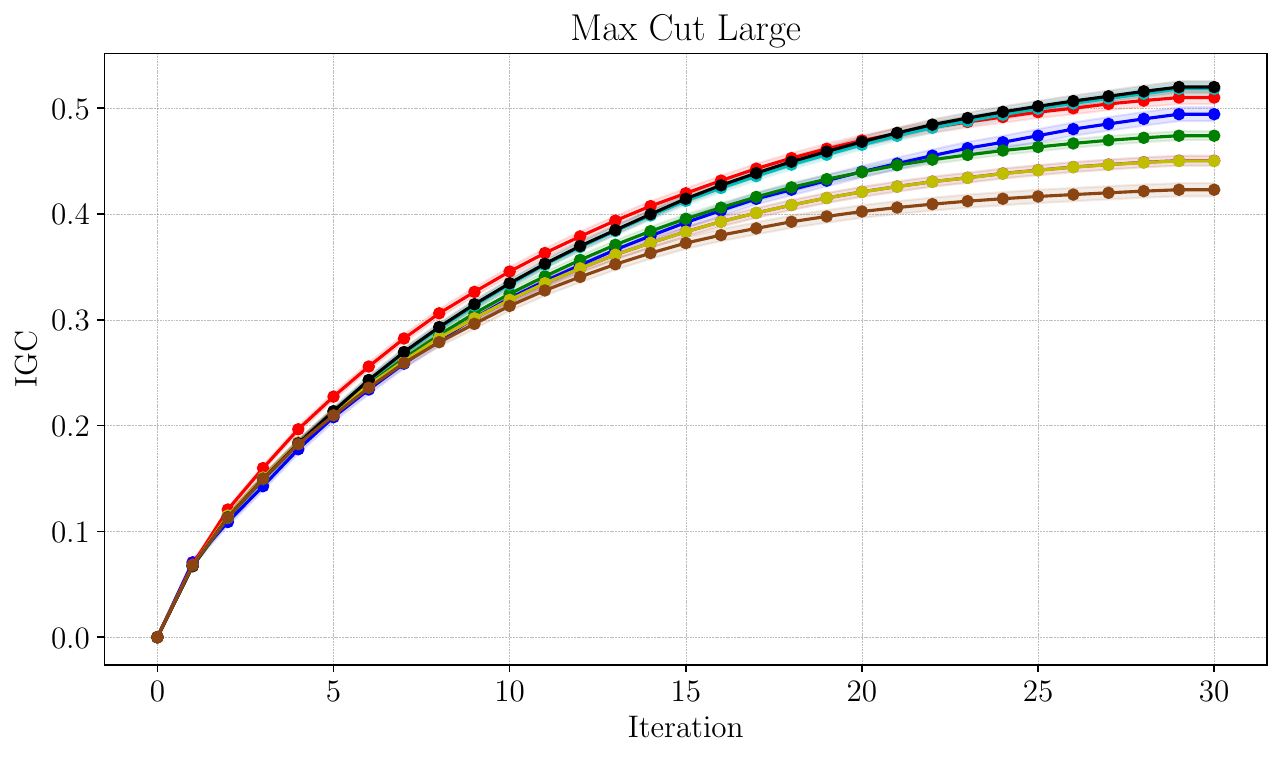}\\
\includegraphics[width=0.33\textwidth]{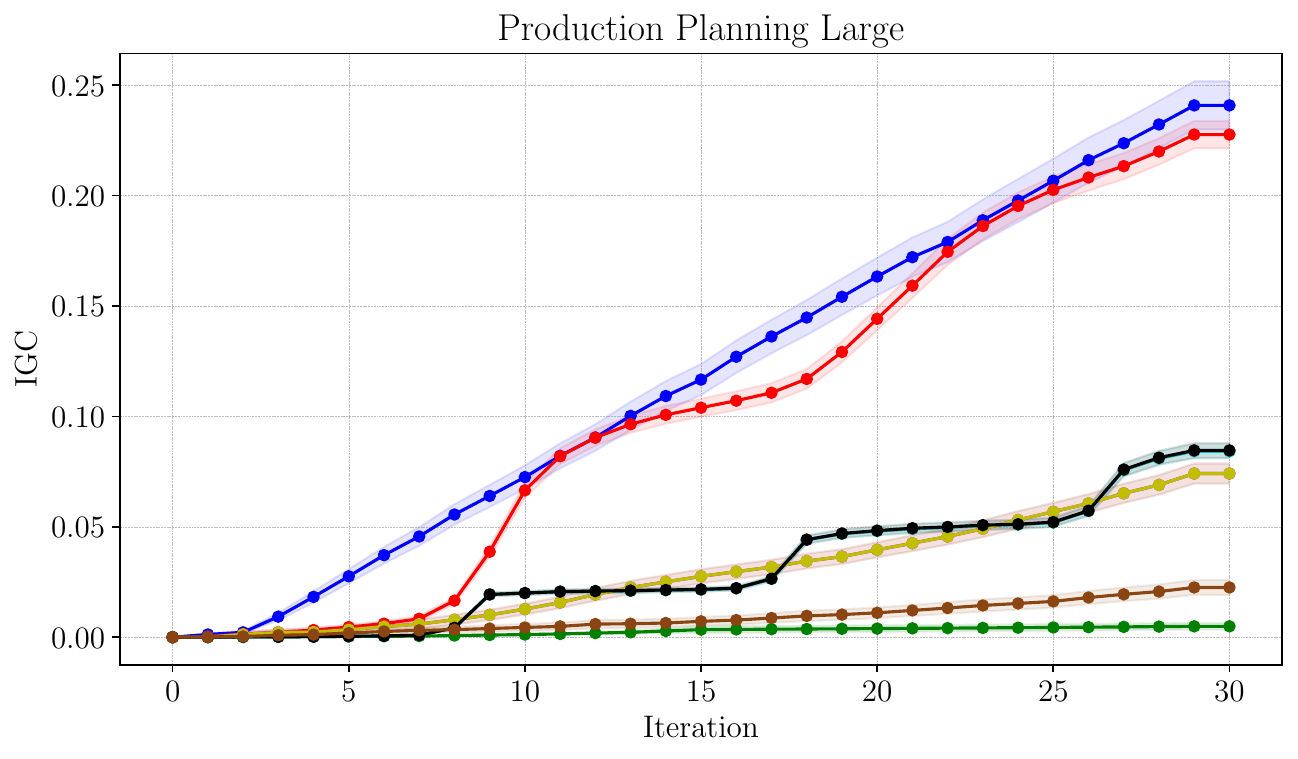}
\includegraphics[width=0.33\textwidth]{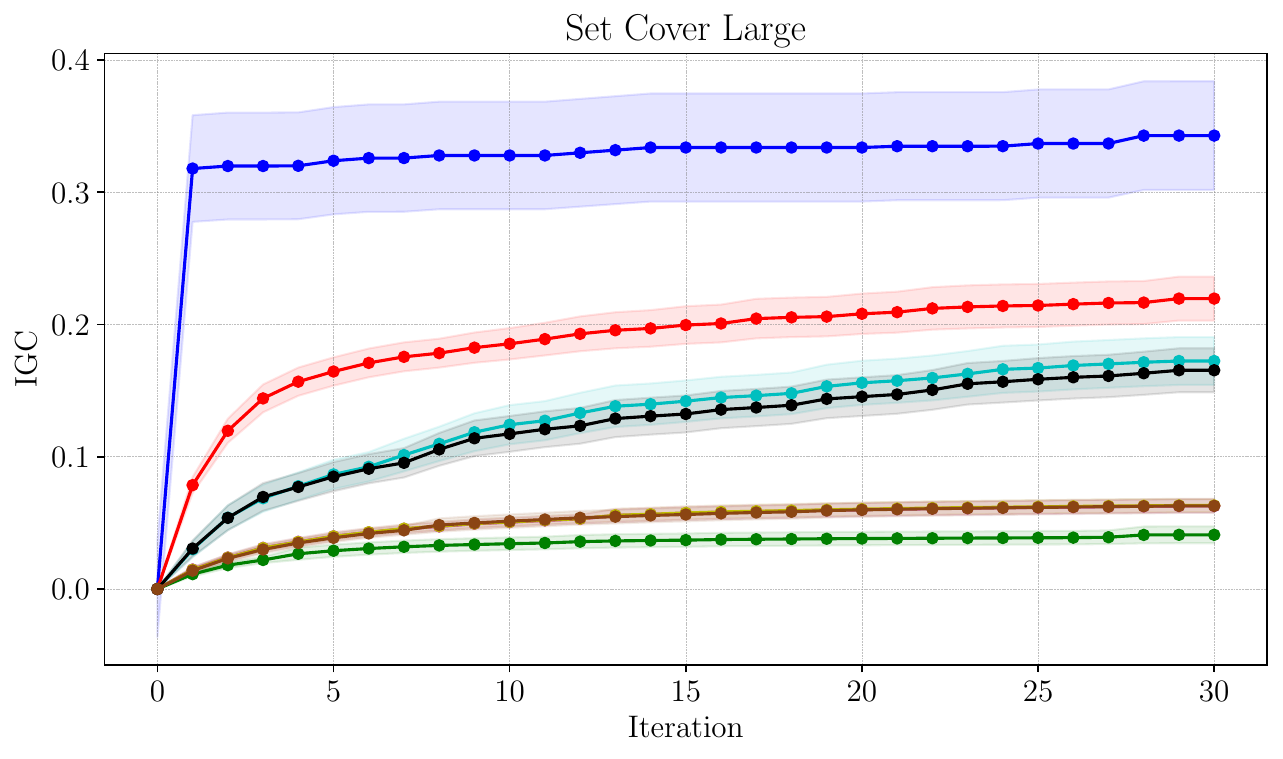}\\
\vspace{-10pt}
\caption{\textbf{Mean IGC for out-of-distribution instances} 
Our cut removal algorithm outperforms or matches all cut addition methods in all benchmarks. Our method shows a much stronger generalization ability onto the larger instances. The scaling improvement is especially accentuated in packing, bin packing and set cover.}
\label{fig:Experiment_2}
\end{figure*}

\subsection{Generalization on Larger Instances}
After having studied the performance that cut removal acting yields we aim to evaluate how well do our trained policies generalize to larger instances. We aim to answer the following questions: Can our cut removal acting algorithm with a model trained on smaller instances generalize for bigger instances? How does it compare with the previous benchmarks?

We remark that generalization to instances of lager size is a very desirable property since the annotation of the data becomes harder as the size increases.

\paragraph{ILP Benchmarks}
We use the same instances as in the first experiment (see subsection \ref{subsec:core-experiment}) to train the models. We run the evaluation with 500 fresh instances of larger problems $\sim50\%$ larger. Details on the specific dimensions are contained in Appendix \ref{subsec:dataset-dimensions}. 

\paragraph{Test Evaluation Metrics}
In order to asses which method is best we use the IGC as in the previous experiment.

\paragraph{Our Experimental Results}
The model $\pi_\theta$ is pretrained according to Section~\ref{subsec:core-experiment}.
Figure \ref{fig:Experiment_2} shows the test IGC for cut removal acting against the benchmark policies for packing, bin-packing, max cut, production planning and set cover.

Even after training $\pi_\theta$ on smaller instances than the ones used for testing, our cut removal algorithm outperforms all cut addition policies in packing, bin-packing, production planning and set cover families. For packing, bin-packing and set cover the margin betweeen cut removal acting and the \textit{look-ahead} policy significantly grows when compared to the margins in training and testing in the medium instances. This shows that the cut removal algorithm scales better than the cut addition benchmarks for this families of instances. For production planning the margin is slightly decreased when compared to the results in the previous experiment as the \textit{look-ahead} policy matches the performance of the cut removal method in the last third of the iterations.

For the max cut family the performance is the same as in the previous experiment. Again, the cut removal acting does not outperform all cut addition methods but the differences in performance are very small and the random policy performs almost as well as the best policies in the first third of the iterations.

\section{Conclusion}
Cutting plane methods play a crucial role in solving Integer Linear Programs. Recent works have employed machine learning techniques to design cut addition methods that, after training, are able to ensure fast convergence to the optimal integral solution~\cite{tang2020reinforcement,paulus2022learning}. In this work, we propose a novel cutting plane method that, at each iteration, introduces multiple cutting planes which are subsequently removed based on the output of a well-trained machine learning model. Our experimental evaluations demonstrate that our method outperforms both human-based heuristics and more recent machine learning-based approaches for cut addition.

Several intriguing research directions emerge for future work. The first involves the use of models capable of capturing the joint combinatorial structure of sets of cuts. The second focuses on designing cutting plane methods that do not maintain a constant increase in the number of linear constraints from iteration to iteration, but rather select the number of linear constraints to add based on an appropriate model. We are confident these approaches have the potential to significantly enhance the convergence properties of modern cutting plane methods.

\section*{Impact Statement}
This paper presents work whose goal is to advance the field of Machine Learning. There are many potential societal consequences of our work, none which we feel must be specifically highlighted here.

\section*{Acknowledgements}
Authors acknowledge the constructive feedback of reviewers and the work of ICML'24 program and area chairs. ARO - Research was sponsored by the Army Research Office and was accomplished under Grant Number W911NF-24-1-0048.
Hasler AI - This work was supported by Hasler Foundation Program: Hasler Responsible AI (project number 21043).
SNF project – Deep Optimisation - This work was supported by the Swiss National Science Foundation (SNSF) under grant number 200021\_205011.
Stratis Skoulakis is supported by Innosuisse through an Innovation project (contract agreement 100.960 IP-ICT)

\bibliography{bibliography}
\bibliographystyle{icml2024}

\newpage
\appendix
\onecolumn
\section*{Contents of the Appendix}
We describe the contents of the supplementary material below:
\begin{itemize}
    \item In Appendix \ref{app:A}, we provide additional details on how the randomized ILP instances are generated and the description on how Gomory Cuts are obtained.
    \item In Appendix \ref{app:B}, we present additional implementation details including the dataset dimensions, definitions for both the baselines and the feature extraction step and the specification for the training hyperparameters.
    \item In Appendix \ref{app:C}, we evaluate the performance of our method in interaction with an ILP solver. We include performance results for both end-to-end and isolated cutting plane stage solving.
    \item In Appendix \ref{app:D}, we discuss runtime considerations and provide a side-by-side comparison on that axis.
    \item In Appendix \ref{app:E}, we explore how the cut quality is distributed in cutpools for each of the different families. We gain some insights on the experimental results in Section \ref{sec:experiments}
\end{itemize}
\section{ILP generation and primer on Gomory Cuts} \label{app:A}
\subsection{Integer Programming Domains} \label{subsec:integer-programming-domains}
We used instances from five integer programming domains. The first four: (i) packing, (ii) bin packing, (iii) maximum cut and (iv) production planning are the ones first suggested by \citet{tang2020reinforcement} and also used by \citet{paulus2022learning}. The exact mathematical formulation for the four first families is given in \citet{tang2020reinforcement}. We extend the benchmarks with instances from the set cover family. 

For set cover, we generate those instances probabilistically. Starting with $m$ elements and $n$ subsets, we add each element to a subset with probability $p$. After the full iteration, we achieve feasibility by ensuring: (i) that no subset is empty by adding a random element to empty subsets, (ii) that all elements are included in at least one subset by adding non-included elements to a random set. Let $E = \{e_1, e_2, \dots, e_n\}$ be a set of $n$ elements. Let $S_1, S_2, \dots, S_m$ be subsets of $E$ with associated costs $c_1, c_2, \dots, c_m$. Let $X_i$ be a random variable associated with each subset $S_i$. $X_i = 1$ if $S_i$ is in the solution, and $0$ otherwise. The ILP formulation is as follows:

\begin{align*}
\text{Minimize} \quad & \sum_{i=1}^{m} c_i X_i \\
\text{subject to} \quad & \forall e \in E, \sum_{i:e\in S_i} X_i \geq 1, \\
& \forall X_i, X_i \in \{0, 1\}.
\end{align*}

The constraints in (1) ensure that every element is present in at least one of the chosen subsets. The constraints in (2) indicate that every subset is either chosen or not. The objective function chooses a feasible solution with the minimum cost. We use $p=0.2$ and $\bm{c}=\bm{1}$ for our experiments.

\subsection{Generating Gomory Cuts} \label{subsec:gomory-cuts}
When an LP is tackled using a simplex algorithm, the primary step involves converting the original LP into standard form. This entails transforming all inequalities into equalities through the introduction of slack variables:

\[
\begin{aligned}
    &\text{minimize} \quad c^T x \\
    &\text{subject to} \quad Ax + Is = b, \\
    &\text{and} \quad x \geq 0, \quad s \geq 0,
\end{aligned}
\]

where $I$ denotes an identity matrix, and $s$ represents the set of slack variables. The simplex method iterates on the tableau formed by $[A, I]$, $b$, and $c$. Upon convergence, the simplex method furnishes a final optimal tableau composed by a constraint matrix $L$ with a constraint vector $v$. A Gomory cut in the standard form space is generated by utilizing the row of the tableau corresponding to a fractional variable in the optimal solution $x^\star$. For each fractional element $x^\star_i$ of $x^\star$ we can generate a Gomory cut
\begin{equation}
    (- L_i + \lfloor L_i \rfloor)^T x \leq - v_i + \lfloor v_i \rfloor,
    \label{eq:cut}
\end{equation}
where $L_i$ is the $i$-th row of the matrix $L$ and $\lfloor \cdot \rfloor$ means component-wise rounding down. We can decompose the generated cuts cutting plane of the following form:

\begin{equation}
    e^T x + r^T s \leq d
    \label{eq:cutting-planes-slack-space}
\end{equation}

where $e, x \in \mathbb{R}^n$, $r, s \in \mathbb{R}^m$, and $d \in \mathbb{R}$. Despite the presence of slack variables, they can be eliminated by multiplying both sides of the linear constraints in \eqref{eq:cutting-planes-slack-space} by $r$:
\begin{equation}
    r^T Ax + r^T s = r^T b    
\end{equation}

and subtracting the new cutting plane (7) from the equation above. This results in an equivalent $\alpha \leq \beta$ cutting plane:

\begin{equation}
    (e^T - r^TA)x \leq d - r^Tb    
\end{equation}

This cutting plane exclusively involves variables within the original variable space. Slack variables do not provide additional information about the polytope and operations for the encoding described in \ref{subsec:feature-encoding} are defined in the original space.

\section{Implementation Details} \label{app:B}
\subsection{Dataset Dimensions} \label{subsec:dataset-dimensions}
\citet{tang2020reinforcement} consider three different sizes (small, medium large) for each domain. In \citet{paulus2022learning} the authors discard small and medium sizes because they are solved at presolving time or after adding a few number of cuts. Although we do not use pre-solving in our study out method has also shown to converge too fast for small and medium instances in \cite{tang2020reinforcement}. 

For the in-distribution experiment we generate $2000$ train, $500$ validation, $500$ test instances of the following dimensions:
\begin{itemize}
    \item Packing: $50$ variables, $50$ (resource) constraints.
    \item Bin Packing: $50$ variables, $50$ (resource) constraints + $50$ binary constraints.
    \item Max Cut: $|V|=9, |E|=25$.
    \item Production Planning: $T=10$.
    \item Set Cover: $|E|=35, |S|=35$.
\end{itemize}
Note that for each of the training and validation instances a trajectory of the $\textit{look-ahead}$ generates up to $30$ cutpools of size around $\frac{m+30}{2}$ where m is the number of constraints which leads to approximately $2000 \cdot 30 \cdot \frac{m+30}{2}$ training datapoints per family. For example, on packing $m=50$ this is approximately $24 \cdot 10^5$ datapoints.

For the generalization into larger instances experiment we generate $500$ test instances of the following dimensions:
\begin{itemize}
    \item Packing: $100$ variables, $100$ (resource) constraints.
    \item Bin Packing: $100$ variables, $100$ (resource) constraints + $100$ binary constraints.
    \item Max Cut: $|V|=14, |E|=40$.
    \item Production Planning: $T=15$.
    \item Set Cover: $|E|=50, |S|=50$.
\end{itemize}

\subsection{Baselines} \label{subsec:baseline-heuristics}
Consider $\mathcal{C}$ to be a cutpool. The baseline heuristics that we use are defined as follows:
\begin{itemize}
    \item Random: Choose $(\alpha_k, \beta_k) \in \mathcal{C}$ uniformly at random.
    \item Max Violation (MV): Let $x^\star$ be the basic feasible solution of the current LP. MV selects the cut corresponding to the maximum fractional component, this is the cut corresponding to the index  $i_s = \mathrm{argmax}_{i}\{|x^\star_i - \text{round}(x^\star_i)|\}$.
    \item Max Normalized Violation (MNV). Recall that $L$ denotes the optimal tableau returned by the simplex algorithm. Let $L_i$ be the $i$th row of $L$. Then, MNV selects the cut corresponding to index $i_s = \mathrm{argmax}_{i}\{|x^\star_i - \text{round}(x^\star_i)| / \Vert L_i \Vert\} $.
   \item Lexicographic: Add the first cut with fractional index $i_s = \mathrm{argmin}_{x^*_i \text{is fractional}}\{i\}$.
   \item Min Similar: Takes the cut $\mathrm{argmin}_{(\alpha_k,\beta_k) \in \mathcal{C}}\{(\alpha_k,\beta_k)^T c\}$ where $c$ is the objective coefficient vector.
\end{itemize}
\citet{WesselmannS12} is a useful resource for a more detailed description of heuristic cut selection rules.

\subsection{Feature Encoding} \label{subsec:feature-encoding}
We design 14 cut features to represent the state for the cut selection task. The first 13 follow \citet{wang2023learning, huang2021learning, WesselmannS12, Achterberg2007, dey}. The 14-th is a binary variable indicating if a cut belongs to the latest cutpool. Table \ref{tab:cut_features} provides a description of such features.

\begin{table}[t]
    \centering
    \caption{Designed cut features for a generated cut $(\alpha_k, \beta_k)$. $c$ is objective coefficient vector and $x^\star$ is the latest LP solution.}
    \resizebox{\textwidth}{!}{
    \begin{tabular}{c|c|c}
    \toprule
         Feature & Description & Number  \\ \midrule
         Cut Coefficients & Mean, Max, Min, Std of $(\alpha_k, \beta_k)$ & 4 \\ \midrule
         Objective Coefficients & Mean, Max, Min, Std of $c$ & 4 \\
         \midrule
         Parallelism & Parallelism between the objective and the cut $\frac{(\alpha_k, \beta_k)^T c}{|c| | (\alpha_k, \beta_k) |}$ & 1 \\ \midrule
         Efficacy & Euclidean distance of the cut hyperplane to $x^\star$ & 1 \\ \midrule
         Support & Proportion of non-zero coefficients of $(\alpha_k, \beta_k)$ & 1 \\ \midrule
         Integral Support & Proportion of non-zero coefficients with respect to integer variables of $(\alpha_k, \beta_k)$ & 1 \\ \midrule
         Normalized Violation & Violation of the cut to the current LP solution $\max \{0, \frac{\alpha_k^T x^\star - \beta_k}{|\beta_k|}\}$ & 1 \\
         \midrule
         Latest Cutpool & Wheter $(\alpha_k, \beta_k) \in \mathcal{C}_k$ or not & 1 \\
         \bottomrule
    \end{tabular}
    }
    \label{tab:cut_features}
\end{table}

\subsection{Training Hyperparameters}
We trained our models with SGD with a lr of $5 \cdot 10^{-3}$ for $50$ epochs using a batch size of $10^4$ with a patience parameter of $5$.

\section{Interfacing with an ILP Solver} \label{app:C}
As remarked in the experiments section \ref{rem:environment} our method is tested in a clean and isolated environment using an implementation of the cutting plane method from scratch. Nevertheless, we acknowledge the interest of evaluating our approach perform inside of an ILP Solver. As a result, we have incorporated our cut-removal approach in the SCIP \cite{Achterberg2007} solver which also enables solving larger instances as well as using other types of cuts implemented natively (Mixed Integer Gomory cuts, Strong Chvátal-Gomory cuts, Complemented Mixed Integer cuts and Implied Bound cuts).

In particular, we have developed an implementation of our method in the SCIP solver through the PySCIPOpt python interface and used it on the “Neural Network Verification” \cite{nair2020solving} dataset instances. More precisely, we evaluated the performance of the Branch-and-Cut mode of SCIP with the vanilla cut-addition policy (B\&C-Cut\_Addition) with the performance of Branch-and-Cut mode with our cut-removal approach (B\&C-Cut\_Removal).

\subsection{End-to-end Performance Comparison}
In Table \ref{tab:improvement} we present the percentage improvement of the solution found by B\&C-Cut\_Removal with respect to the solution found by B\&C-Cut\_Addition in 26 randomly selected instances with the node limit set to 100. Our experiments reveal that B\&C-Cut\_Removal finds on average a $35\%$ better solution than B\&C-Cut\_Addition. We also observe that B\&C-Cut\_Removal is able to find a better solution on $88.46\%$ of instances. In the remainder $11.54\%$ of instances both methods reach the same solution.

\begin{table}[htbp]
\centering
\caption{Improvement (\%) for Each Instance (id)}
\label{tab:improvement}
\setlength{\tabcolsep}{4pt}
\begin{tabular}{cccccccccccccc}
\toprule
 Instance & 1317 & 1891 & 1941 & 1987 & 2229 & 2891 & 2959 & 321 & 3736 & 3853 & 3964 & 4173 & 4329 \\
\midrule
Imp. (\%) & 27.87 & 4.69 & 0.00 & 125.80 & 46.38 & 76.00 & 22.55 & 17.32 & 22.75 & 74.23 & 20.69 & 0.00 & 3.52 \\
\midrule
\midrule
 Instance & 4743 & 495 & 5119 & 5424 & 5463 & 5757 & 5833 & 6392 & 6481 & 7064 & 8509 & 8627 & 8630 \\
\midrule
Imp. (\%) & 0.00 & 47.85 & 25.24 & 33.36 & 21.32 & 52.88 & 195.50 & 34.01 & 19.20 & 12.12 & 11.36 & 12.87 & 12.91 \\
\bottomrule
\end{tabular}
\end{table}

\subsection{Isolated Performance Comparison}
Branching algorithms can be interpreted by the tree they describe. The starting problem formulation relies on the top node of a tree. After making a branching decision, children problem formulations with more restrictive constraints appear. At each of this sub-problem formulations (nodes) the Cutting Plane procedure is invoked, this is where our Cut Removal Algorithm is executed.

We present a second experiment in order to evaluate the improvement that our method yields specifically at a single-node level in the Cutting-Plane stage of the SCIP solver by isolating all parts of the SCIP workflow except the cutting plane step. This setting is analogous to our Cutting Plane implementation but in SCIP. Interfacing with SCIP allows for extended capabilities such as having different kinds of cuts (besides vanilla Gomory Cuts) and having the problem modified through pre-solving at the starting node. At each Cutting Plane method call we measure the improvement of the LP bound of B\&C-Cut\_Removal against B\&C-Cut\_Addition. The results are contained in Table \ref{tab:mean_improvement}

\begin{table}[htbp]
\centering
\caption{Mean Improvement (\%) for Each Instance (id)}
\label{tab:mean_improvement}
\setlength{\tabcolsep}{4pt}
\begin{tabular}{cccccccccccccc}
\toprule
Instance & 8630 & 2891 & 3853 & 5424 & 4329 & 8627 & 6481 & 1941 & 3964 & 495 & 2959 & 4173 & 8509 \\
\midrule
MImp. (\%) & 137.83 & 29.69 & 0.09 & 0.69 & 109.31 & 276.44 & 3.09 & 0.01 & 47.71 & 0.37 & 22.20 & 0.00 & 31.97 \\
\midrule
\midrule
Instance & 7064 & 1317 & 5463 & 2229 & 5757 & 1891 & 4743 & 3736 & 1987 & 321 & 5833 & 5119 & 6392 \\
\midrule
MImp. (\%) & 91.77 & 117.55 & 322.77 & 59.97 & 38.47 & 0.51 & 2.44 & 0.00 & 8.21 & 1.78 & 9.61 & 102.77 & 35.61 \\
\bottomrule
\end{tabular}
\end{table}

\section{Runtime Considerations}\label{app:D}
Wall-clock time highly depends on external factors such as implementation and programming language (e.g., C++ vs Python/PyTorch). This variability is why the number of cutting planes (number of iterations) is considered a more robust evaluation metric and has also been adopted by \citet{paulus2022learning, tang2020reinforcement}. This being said, we acknowledge the interest in providing results for this metric. In Table \ref{tab:runtime}, we present the speedup that our method achieves when compared to the baselines in terms of wall-clock time. We measure the total time elapsed from start to finish when solving the instance. We consider instances where at least one of the policies converged. Max-Cut and Planning are not included in this table as none of the methods were able to achieve an ICG value of 1 (see Figure 2 in the paper). Thus, these problems do not permit a fair comparison. We notice our method attains an improvement over all baselines.

\begin{table}[htbp]
\centering
\caption{Performance Comparison}
\label{tab:runtime}
\begin{tabular}{lccc}
\toprule
Problem/Method & Paulus et al. & Look-ahead Expert & Other Non-neural Baselines \\
\midrule
Bin Packing & 1.14x & 24.59x & 2.71x \\
Packing & 1.61x & 24.30x & 2.15x \\
Set Cover & 5.12x & 47.26x & 5.41x \\
\bottomrule
\end{tabular}
\end{table}

We emphasize that our current implementation on cut removal could be further optimized with data structures. For this reason, we believe that the reported speedups could be improved further.

\section{Cut Pool Distributions} \label{app:E}
In this section we present results on the quality distribution for Gomory Cuts. We aim to answer the following questions: How is the quality in cuts distributed for different ILP families? How does this distribution evolve across the iterations? How do the cutpool quality distributions vary if acting with a random policy versus the \textit{lookahead} rule?

\paragraph{MILP Benchmarks} We collect trajectories of the random policy and \textit{lookahead} policy for a total of 500 instances per problem family. At each iteration $k$ we save each of the generated cuts $(\alpha, \beta)$, the previous LP value $x^\star_k$ and we calculate the LP value after adding the cut $x^\star_{k\cup (\alpha, \beta)}$.

\paragraph{Test Evaluation Metrics}
In section \ref{sec:methodology} we presented the bound improvement as a criteria to measure the quality of a cut $(\alpha, \beta)$.

Recall the formulation for the \textit{normalized LP improvement} by the previous LP value: 
\[\frac{c^T x^\star_{k\cup (\alpha,\beta)} - c^T x^\star_{k}}{c^T x^\star_{k}}.\]
This metric (M1) serves as an indicator on how valuable is a cut for the solution of the LP.

Consider also the similar construction but normalizing by the largest bound improvement instead.
\[\frac{c^T x^\star_{k\cup (\alpha,\beta)} - c^T x^\star_{k}}{\max_{(\tilde{\alpha},\tilde{\beta}) \in \mathcal{C}_k}\{c^T x^\star_{k\cup (\tilde{\alpha},\tilde{\beta})} - c^T x^\star_{k}\}} \in [0,1].\]
This metric (M2) serves as an indicator on how valuable is a cut with respect to its cutpool.

\paragraph{Results} For each problem family (packing, bin packing, maximum cut, production planning, set cover), each metric (M1, M2) and each policy generating the trajectory (random, \textit{look-ahead}) we compute distribution matrix $D$. Consider M(·) to be a function that calculates metric M for each element of an array and $\text{sortd}$(·) to be a function that sorts an array in decreasing order. Let $\mathcal{C}_k^{(l)}$ to be the cutpool generated at iteration $k$ for instance $l$. Then, the i-th row of $D$ is calculated by aggregating ${\text{sortd}(M(C_i^{(l)}))}$ and scaling. For the scaling note that for some problems the cutpool dimensions may vary across different instances and iterations. For this reason, we divide each component by the number of existing components in the cutpools in the aggregation. Figures \ref{fig:packing-distribution}, \ref{fig:binpacking-distribution}, \ref{fig:maxcut-distribution}, \ref{fig:planning-distribution}, \ref{fig:setcover-distribution} show the heatmaps for the different distribution matrices. 

The top plots (a), (b) for each figure explain how the "cut quality with respect to fellow cutpool cuts" distribution evolves across iterations for trajectories of the random and \textit{look-ahead} policies respectively. The starting cutpools are uneven for all problems and that picking the best cuts in the first iterations (\textit{look-ahead}) leads to more uniform cutpools compared with random picking in all problems except production planning where the uneven distribution holds.

The bottom plots (c), (d) show how the "cut quality with respect to previous LP value" distribution evolves across iterations for trajectories of the random and \textit{look-ahead} policies respectively. For trajectories of both policies most of the bound improvement with respect to the previous solution is yielded in the first iterations except for production planning where the magnitudes hold. This explains the results observed in Figures \ref{fig:Experiment_1}, \ref{fig:Experiment_2} where the IGC plateaus after the first iterations for all problems except for production planning where it shows a linear trend. We also observe that for the max cut the number of high quality cuts in the first iterations is significantly larger than in other families. This justifies the behaviour observed in the maxcut plot in Figures \ref{fig:Experiment_1}, \ref{fig:Experiment_2} where for the first iterations many policies perform as well as the best one.

\begin{figure}[ht]
    \centering

    \subfigure[]{
        \includegraphics[width=0.45\linewidth]{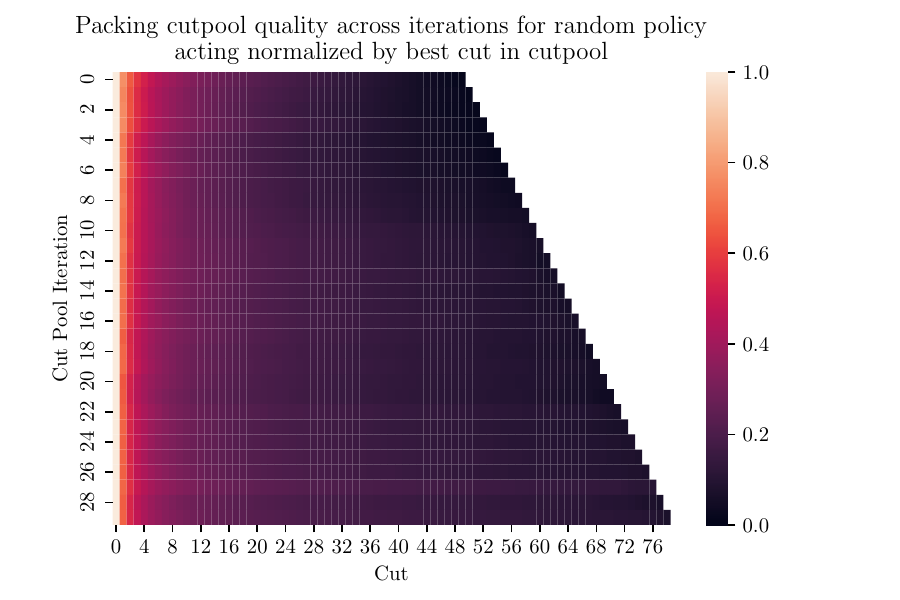}
        \label{fig:packing-distribution-random-maxnorm}
    }
    \hfill
    \subfigure[]{
        \includegraphics[width=0.45\linewidth]{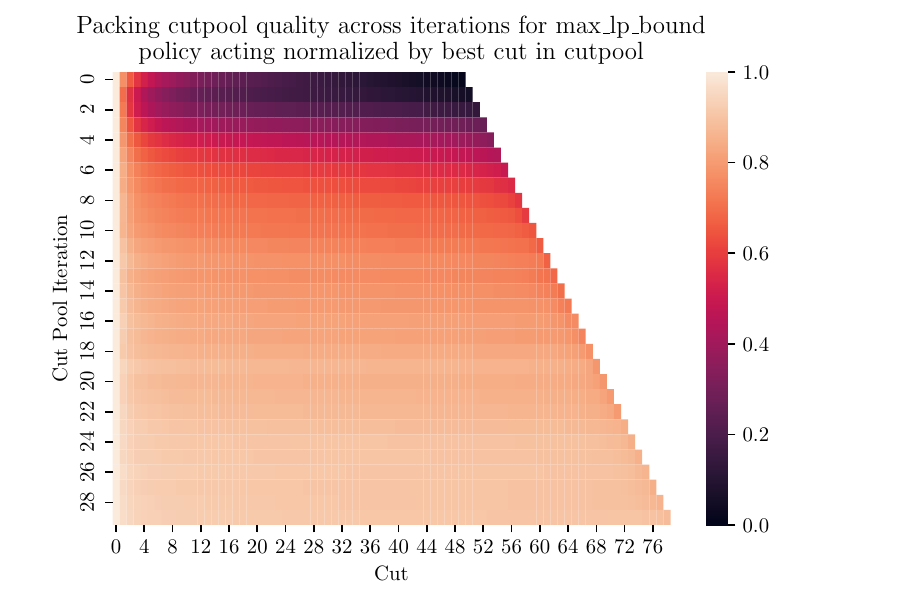}
        \label{fig:packing-distribution-maxlp-maxnorm}
    }

    \subfigure[]{
        \includegraphics[width=0.45\linewidth]{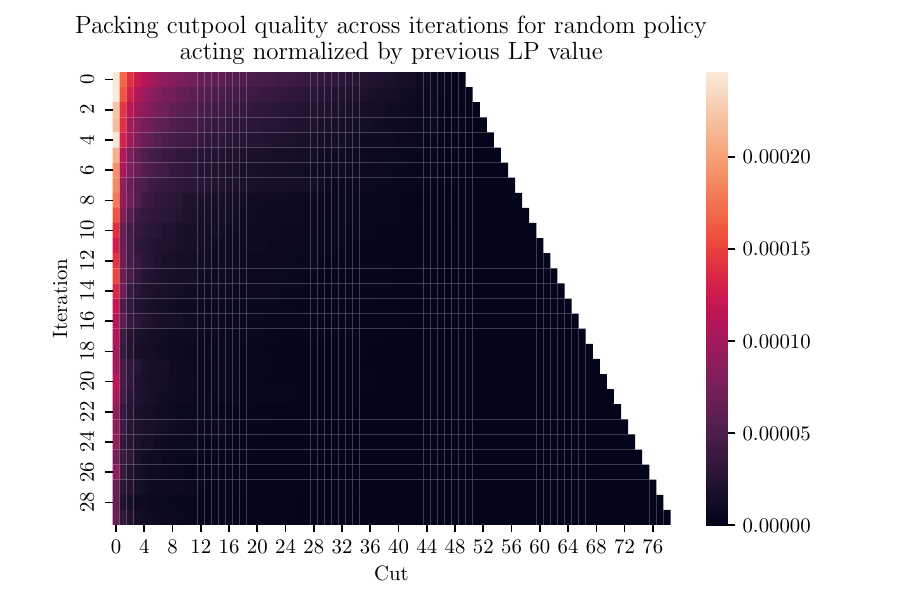}
        \label{fig:packing-distribution-random}
    }
    \hfill
    \subfigure[]{
        \includegraphics[width=0.45\linewidth]{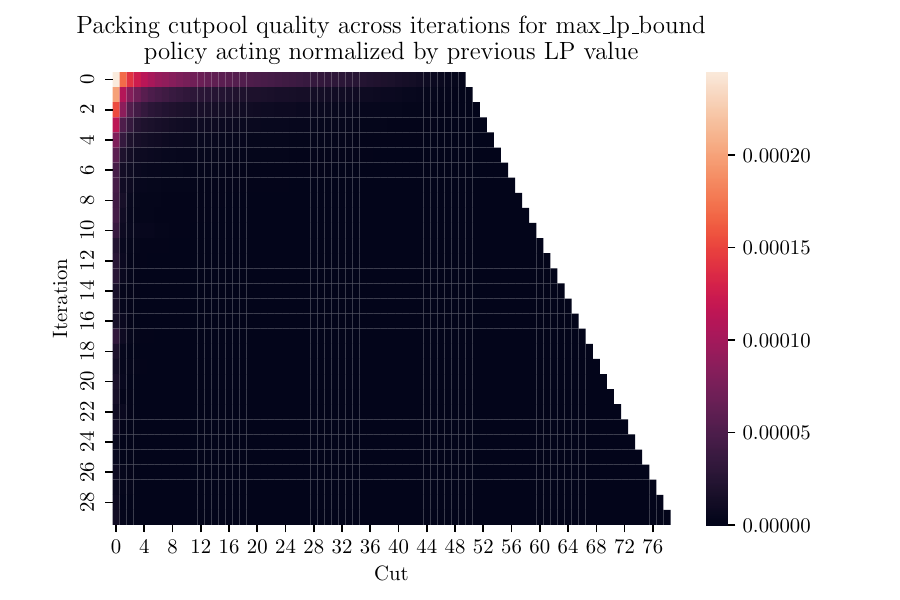}
        \label{fig:packing-distribution-maxlp}
    }

    \caption{Packing Cutpool Distributions}
    \label{fig:packing-distribution}
\end{figure}

\begin{figure}[ht]
    \centering

    \subfigure[]{
        \includegraphics[width=0.45\linewidth]{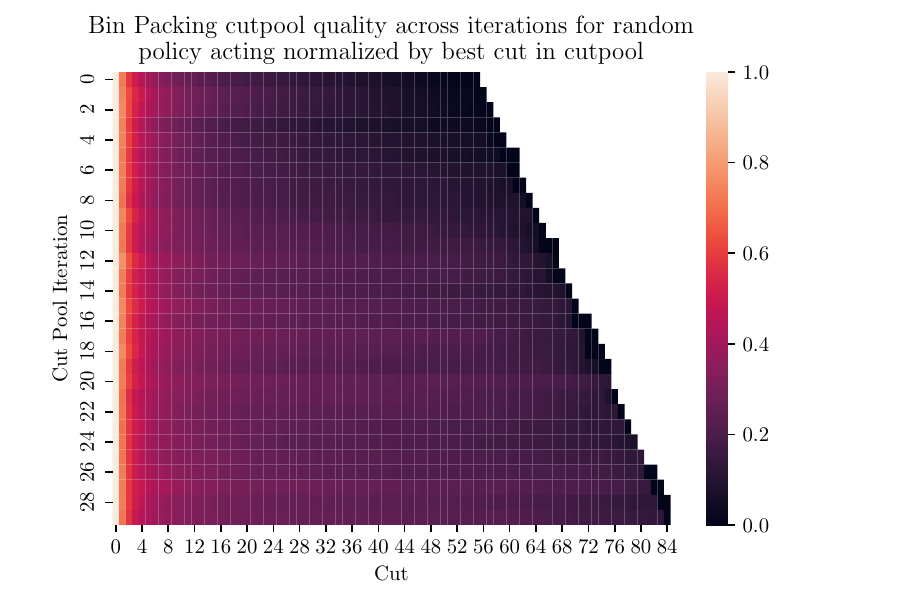}
        \label{fig:binpacking-distribution-random-maxnorm}
    }
    \hfill
    \subfigure[]{
        \includegraphics[width=0.45\linewidth]{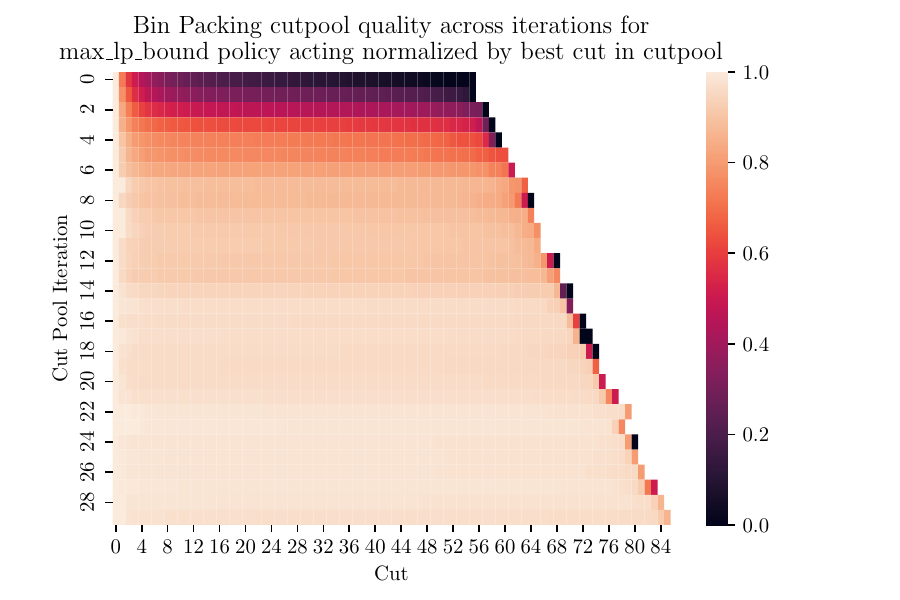}
        \label{fig:binpacking-distribution-maxlp-maxnorm}
    }

    \subfigure[]{
        \includegraphics[width=0.45\linewidth]{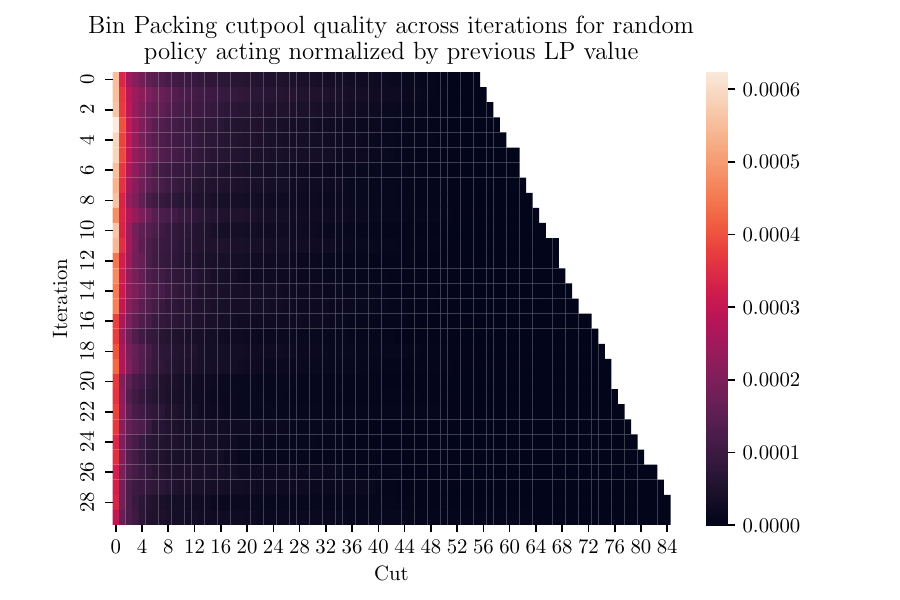}
        \label{fig:binpacking-distribution-random}
    }
    \hfill
    \subfigure[]{
        \includegraphics[width=0.45\linewidth]{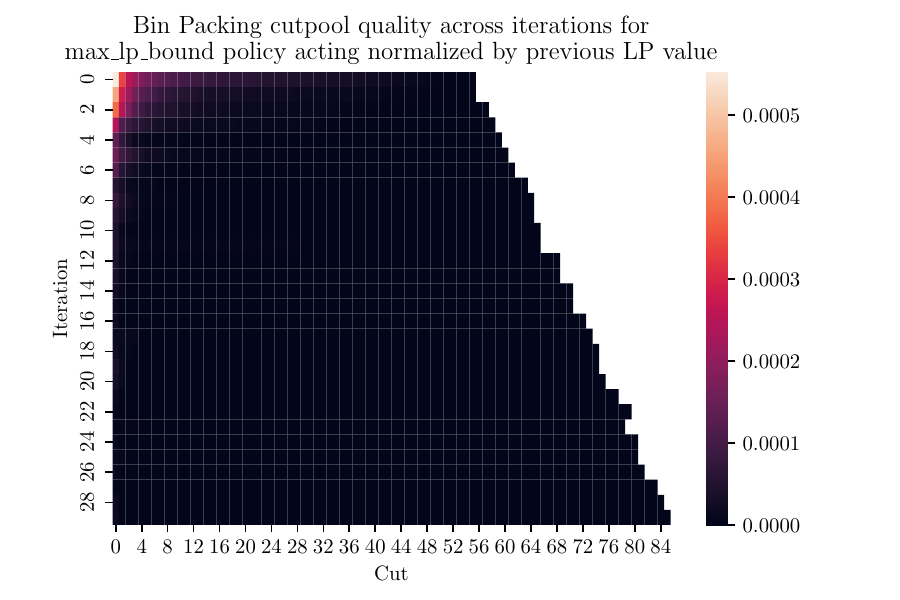}
        \label{fig:binpacking-distribution-maxlp}
    }

    \caption{Bin Packing Cutpool Distributions}
    \label{fig:binpacking-distribution}
\end{figure}

\begin{figure}[ht]
    \centering

    \subfigure[]{
        \includegraphics[width=0.45\linewidth]{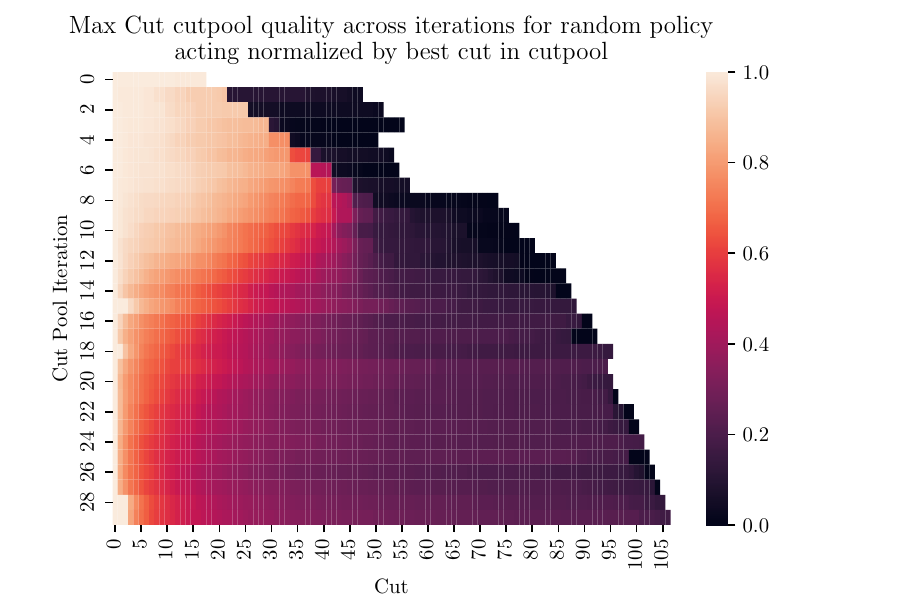}
        \label{fig:maxcut-distribution-random-maxnorm}
    }
    \hfill
    \subfigure[]{
        \includegraphics[width=0.45\linewidth]{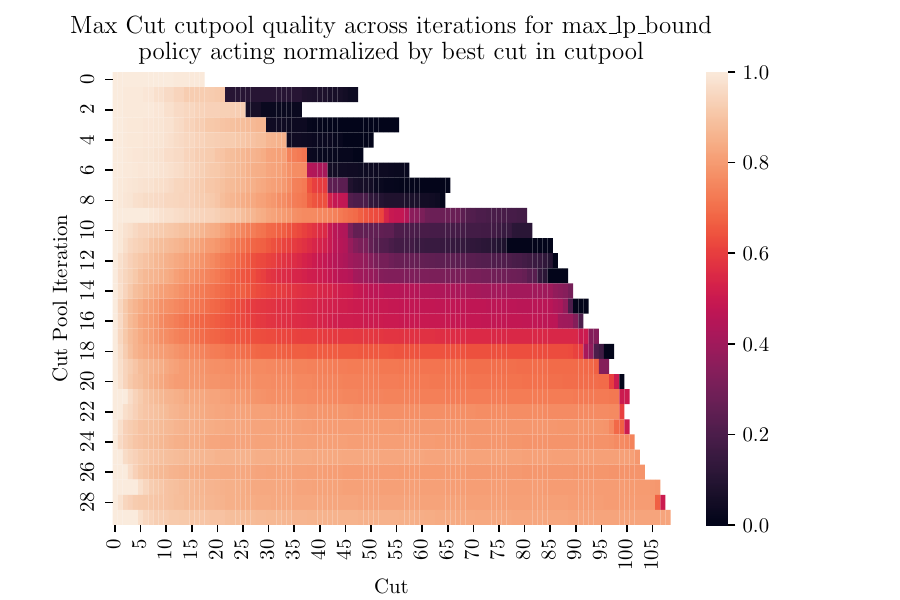}
        \label{fig:maxcut-distribution-maxlp-maxnorm}
    }

    \subfigure[]{
        \includegraphics[width=0.45\linewidth]{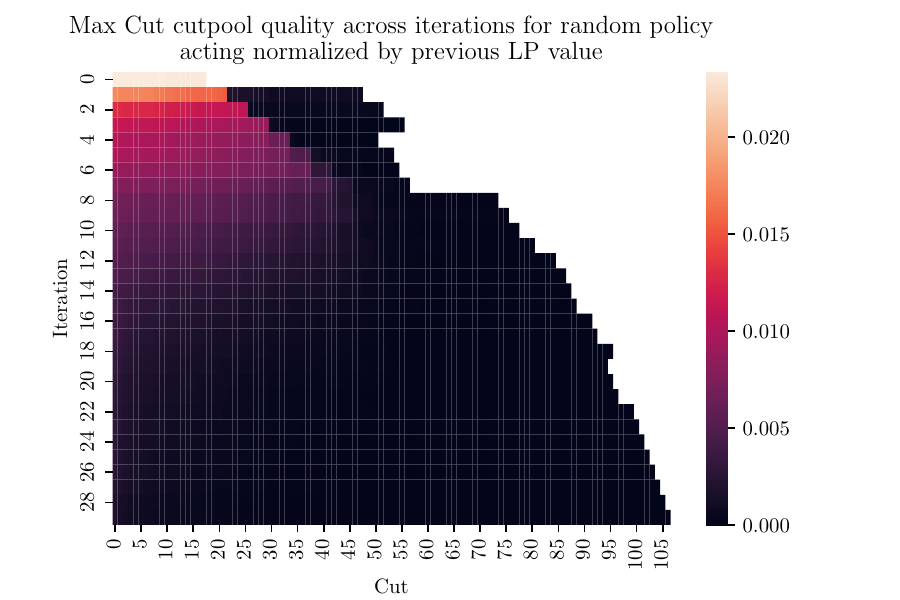}
        \label{fig:maxcut-distribution-random}
    }
    \hfill
    \subfigure[]{
        \includegraphics[width=0.45\linewidth]{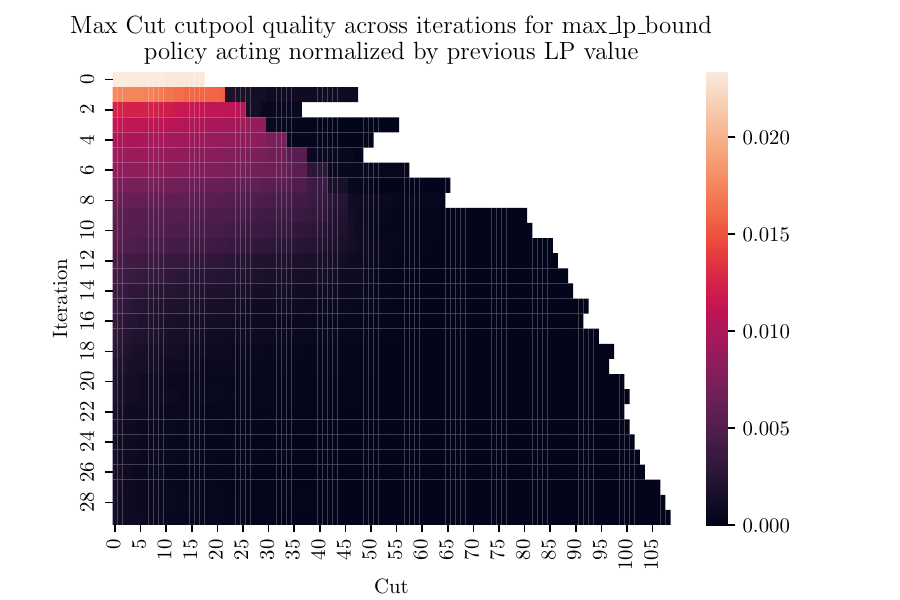}
        \label{fig:maxcut-distribution-maxlp}
    }

    \caption{Max Cut Cutpool Distributions}
    \label{fig:maxcut-distribution}
\end{figure}

\begin{figure}[ht]
    \centering

    \subfigure[]{
        \includegraphics[width=0.45\linewidth]{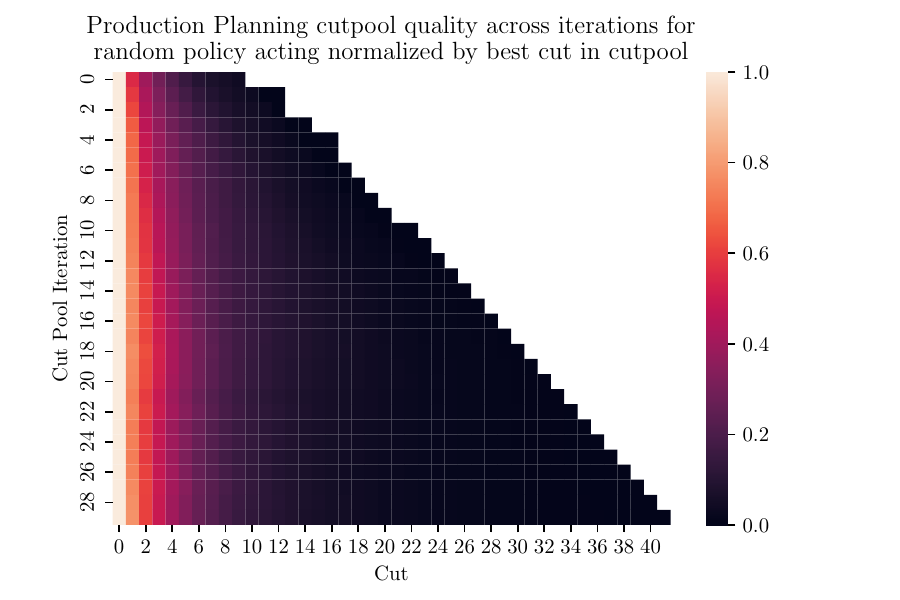}
        \label{fig:planning-distribution-random-maxnorm}
    }
    \hfill
    \subfigure[]{
        \includegraphics[width=0.45\linewidth]{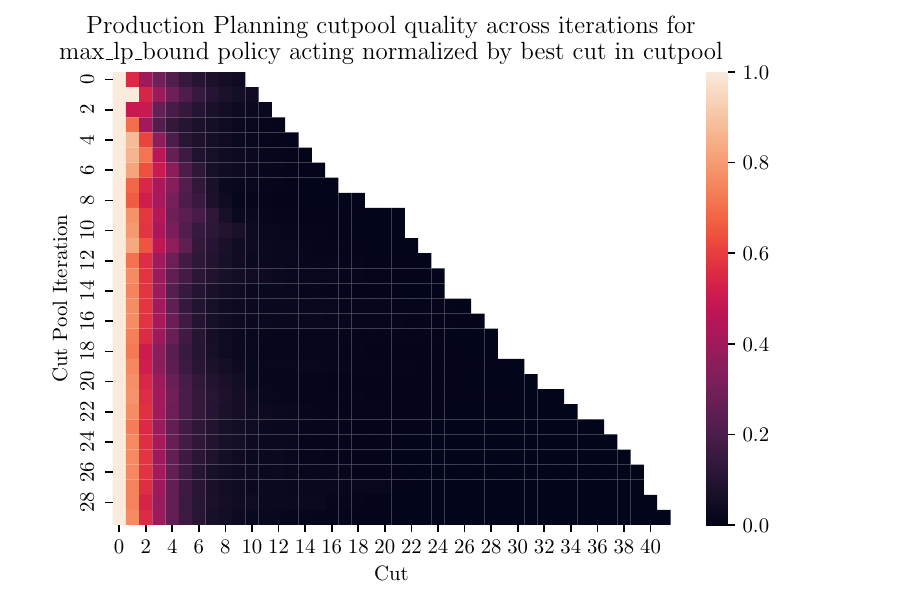}
        \label{fig:planning-distribution-maxlp-maxnorm}
    }

    \subfigure[]{
        \includegraphics[width=0.45\linewidth]{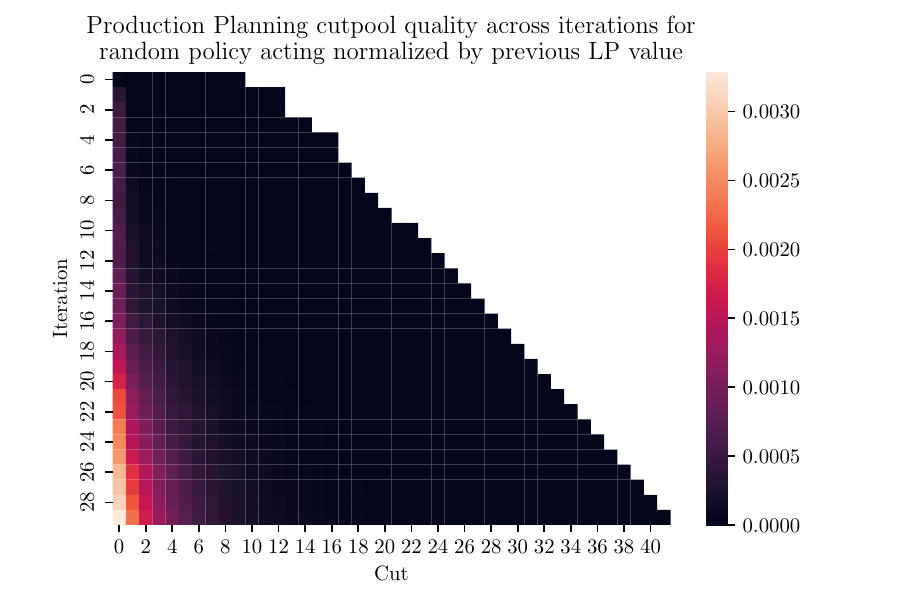}
        \label{fig:planning-distribution-random}
    }
    \hfill
    \subfigure[]{
        \includegraphics[width=0.45\linewidth]{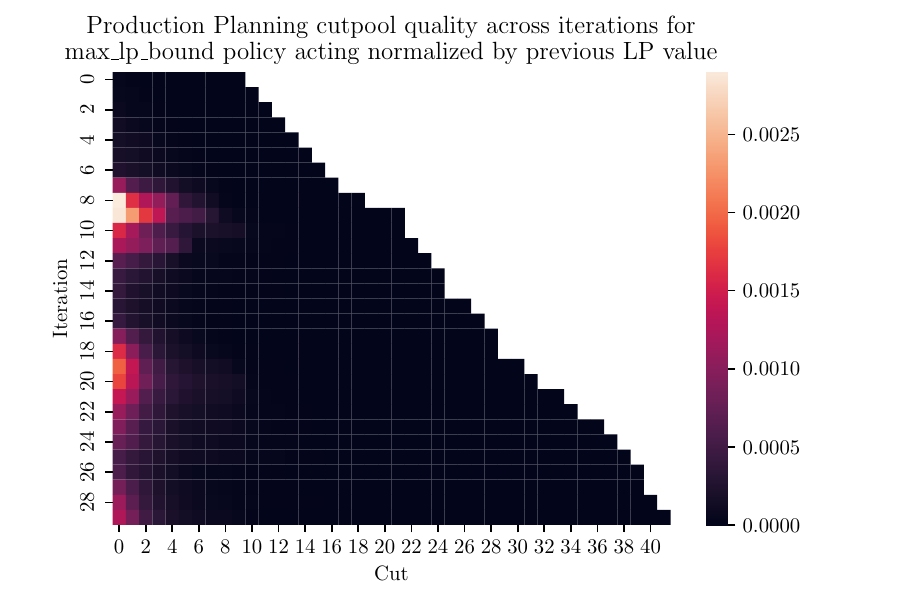}
        \label{fig:planning-distribution-maxlp}
    }

    \caption{Production Planning Cutpool Distributions}
    \label{fig:planning-distribution}
\end{figure}

\begin{figure}[ht]
    \centering

    \subfigure[]{
        \includegraphics[width=0.45\linewidth]{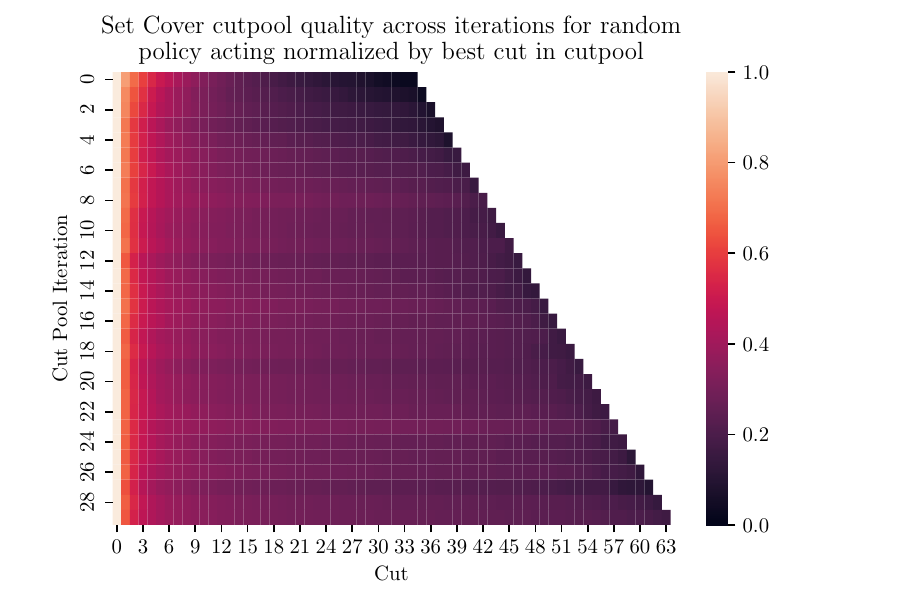}
        \label{fig:setcover-distribution-random-maxnorm}
    }
    \hfill
    \subfigure[]{
        \includegraphics[width=0.45\linewidth]{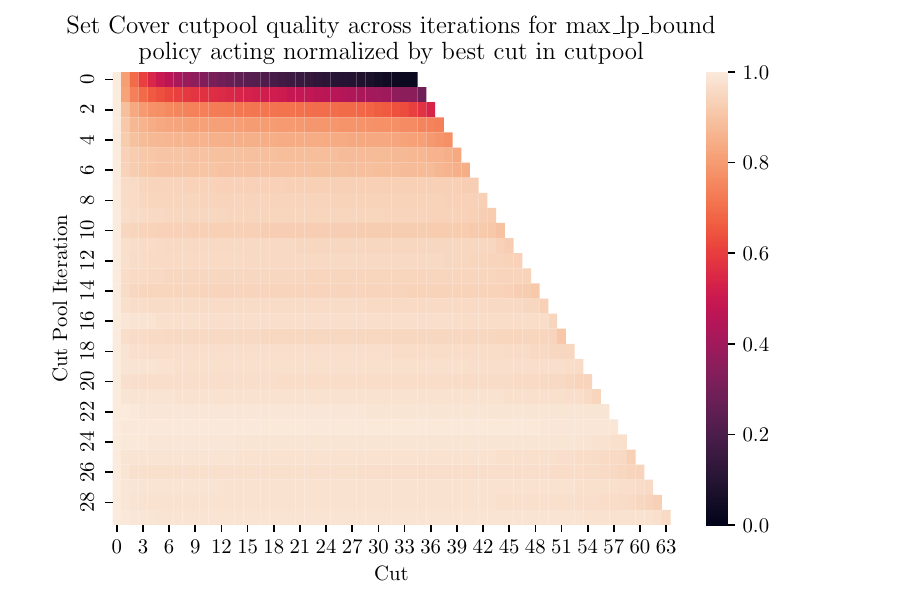}
        \label{fig:setcover-distribution-maxlp-maxnorm}
    }

    \subfigure[]{
        \includegraphics[width=0.45\linewidth]{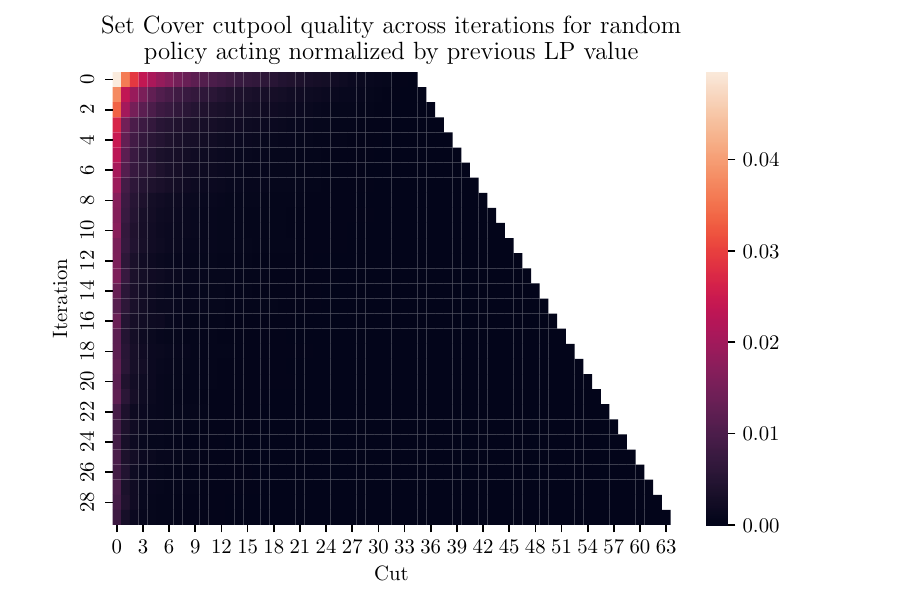}
        \label{fig:setcover-distribution-random}
    }
    \hfill
    \subfigure[]{
        \includegraphics[width=0.45\linewidth]{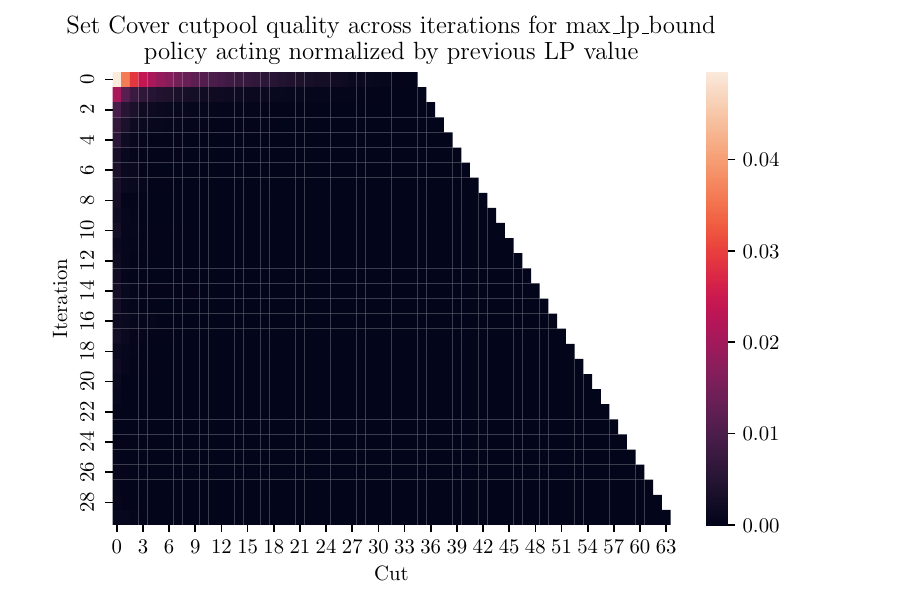}
        \label{fig:setcover-distribution-maxlp}
    }

    \caption{Set Cover Cutpool Distributions}
    \label{fig:setcover-distribution}
\end{figure}

\end{document}